\documentclass[10pt]{article}
\usepackage{amsfonts, epsfig, amsmath, amssymb, color}
\usepackage{textcomp}
\usepackage{paralist}
\usepackage{selinput}
\usepackage{ulem}
\usepackage{subfig}
\usepackage[english]{babel}

\renewcommand {\epsilon}{\varepsilon}

\newcommand{\DD}{\mathbb{D}}
\newcommand{\EE}{\mathbb{E}}

\newcommand{\NN}{\mathbb{N}}

\newcommand{\PP}{\mathbb{P}}

\newcommand{\RR}{\mathbb{R}}

\newcommand{\bB}{\mathcal{B}}
\newcommand{\cC}{\mathcal{C}}

\newcommand{\nN}{\mathcal{N}}

\newcommand{\wW}{\mathcal{W}}

\newcommand{\fB}{\mathfrak{B}}
\newcommand{\fC}{\mathfrak{C}}

\newcommand{\al}{\alpha}

\newcommand{\la}{\lambda}

\newcommand{\om}{\omega}
\newcommand{\Om}{\Omega}

\newcommand{\lv}{\langle}
\newcommand{\rv}{\rangle}
\newcommand{\pd}{\partial}

\newcommand{\ra}{\rightarrow}

\newcommand{\lra}{\longrightarrow}
\newcommand{\ti}{\tilde}

\newcommand{\vzv}{\Leftrightarrow}

\newcommand{\ds}{\displaystyle}
\newcommand{\ind}{\mathbf{1}}

\newcommand{\lqq}{\leqslant}
\newcommand{\gqq}{\geqslant}

\newtheorem{thms}{Theorem}[section]

\newtheorem{defn}[thms]{Definition}

\newtheorem{thm}{Theorem}[section]

\DeclareMathSymbol{\ophi}{\mathalpha}{letters}{"1E}

\renewcommand{\phi}{\varphi}

\newcommand{\be}{\begin{equation}}
\newcommand{\ee}{\end{equation}}
\newcommand{\ben}{\begin{equation*}}
\newcommand{\een}{\end{equation*}}

\newcommand{\ba}{\begin{equation}\begin{aligned}}
\newcommand{\ea}{\end{aligned}\end{equation}}

\DeclareMathOperator{\wnq}{\mathrm{w}_{n}}

\DeclareMathOperator{\wns}{\mathrm{w}_{n}^*}
\DeclareMathOperator{\wnt}{\widetilde{\mathrm{w}_{n}}}

\allowdisplaybreaks[4]

\newenvironment{proof}{\par\noindent{\bf Proof:}}{\hfill$\blacksquare$\par}

\newfont{\cyrfnt}{wncyr10}
\def\J3{\cyrfnt{\rm \u{\cyrfnt I}}}
\def\j3{\cyrfnt{\rm \u{\cyrfnt i}}}

\usepackage[]{color}
\definecolor{DarkGreen}{rgb}{0.1,0.7,0.3}   

\definecolor{DarkGreen}{rgb}{0.1,0.7,0.3}   

\allowdisplaybreaks[4]
\date{\today}

\begin{document}
\thispagestyle{plain}

\title{
How close are time series to power tail L\'evy diffusions? 
}

\author{
Jan M. Gairing
\footnote{Institut f\"ur Mathematik, Humboldt-Universit\"at zu Berlin, Germany; gairing@math.hu-berlin.de} 
\hspace{0.5cm}
Michael A. H\"ogele
\footnote{Departamento de Matem\'aticas, Universidad de los Andes, Bogot\'a, Colombia; ma.hoegele@uniandes.edu.co}
\hspace{0.5cm}
Tania Kosenkova 
\footnote{Institut f\"ur Mathematik, Universit\"at Potsdam, Germany; kosenkova@math.uni-potsdam.de} 
\hspace{0.5cm}
Adam H. Monahan
\footnote{School of Earth and Ocean Sciences, University of Victoria, Victoria BC, Canada; monahana@uvic.ca}
}

 \maketitle

\begin{abstract}
This article presents a new and easily implementable 
method to quantify the so-called coupling distance 
between the law of a 
time series and the law of a differential equation 
driven by Markovian additive jump noise with heavy-tailed jumps, 
such as $\alpha$-stable L\'evy flights. 
Coupling distances measure the proximity 
of the empirical law of the tails of the jump increments 
and a given power law distribution. 
In particular they yield an upper bound for 
the distance of the respective laws on path space. 
We prove rates of convergence comparable 
to the rates of the central limit theorem 
which are confirmed by numerical simulations. 
Our method applied to a paleoclimate time series of glacial climate variability confirms 
its heavy tail behavior. In addition this approach gives evidence 
for heavy tails in data sets of precipitable water vapor 
of the Western Tropical Pacific. 
\end{abstract}


\section{Introduction}\label{sec: intro}

 Understanding the nature of noise in data is of paramount interest in nonlinear dynamics 
 in order to build statistically trustworthy models, whose study leads to relevant new theoretical insights. 
 In this context stochastic modeling consists in the study of deterministic and stochastic differential equations, 
 which represent on the one hand the mechanistic understanding of the underlying phenomenon and 
 on the other different kinds of noise accounting the effect of unresolved processes. 
 In various applications the fluctuations present in time series of interest 
 exceed the plausible thresholds for an underlying Gaussian white noise perturbation.
 The continuity of Markovian Gaussian models given as stochastic differential equations 
 makes it necessary to go beyond the Gaussian 
 paradigm to model random fluctuations and to include the effect of shocks. 
 A natural class of Markovian perturbations containing discontinuous paths 
 is given by so-called L\'evy processes, sometimes referred to as L\'evy flights, 
 which are non-Gaussian and discontinuous extensions of Brownian motion 
 which retain the white noise structure of stationary, $\delta$-correlated increments. 
 While some such processes enjoy scaling properties 
 (for instance $\alpha$-stable processes),  
 scaling is not a generic property of L\'evy processes, (Section \ref{sec: Levy and coupling}). 
 
 In this article we present a new statistical method to estimate the proximity of 
 time series to the paths of a given diffusion model with additive jump L\'evy noise. 
 In essence this is the comparison (in $L^2$-sense) 
 of a discrete approximation of a path and ``typical'' paths of such a model. 
 We stress that the path-wise comparison 
 is the strongest possible form of comparison between stochastic processes. 
 Our method consists of two steps. 
 
 The first step in this method uses  
 our main theoretical result (provided in \cite{GHKK14}), which states that the distance 
 of the paths can be estimated using so-called coupling distances 
 between the underlying L\'evy jump measures.
 This distance, which compares 
 the distributions of \textit{instantaneous} jump increments of the processes, 
 is based on appropriately scaled Wasserstein distances. Wasserstein distances between two probability distributions 
 are well-known in the mathematics literature \cite{RR1998} and measure the ``optimal $L^2$ transport'' between 
 the two distributions. In other words, they minimize on the product space 
 the deviation from the diagonal (in $L^2$-sense) over all joint distributions of the two ``marginal'' distributions. 
 We emphasize that an estimate of the distance 
 of temporal path statistics using a distance only of the instantaneous jump distributions 
 is not entirely straight-forward. 
 First, there are in general infinitely many 
 of such jumps in any finite time interval and the deterministic 
 motion does not mitigate the effect of jumps or can even enhance it. 
 Furthermore, even in the case of finitely many jumps, the jumps of two processes 
 beyond a common threshold typically occur with different intensities.  
 In order to compare such paths with different ``random clocks'' 
 we have to ``synchronize'' them in order to compare the 
 jump distributions. 
 As a toy example of our result 
 imagine two compound Poisson processes with joint intensity $\la>0$, 
 with different jump distributions $\mu_1$ and $\mu_2$. 
 Then our main result states that on a finite time horizon 
 the path statistics of these two processes (in the $L^2$-norm of the supremum distance) 
 can be estimated using the difference of the coupling distance of $\mu_1$ and $\mu_2$. 

 The second step in our method exploits the following advantage of the notion of a coupling-distance.  
 Since its main ingredients are the so-called Wasserstein distances, 
 it is statistically tractable in that we can use well-established 
 results on rates of convergence. 
 In \cite{GHKK16} the notion of coupling distance is shown to be weak enough (in a topological sense) 
 to allow for an efficient statistical assessment of a given time series within a class of models. 
 In this study we prove the most relevant version 
 of the original result on the rates of convergence 
 and illustrate these results by using simulations. 
 In Section \ref{sec: simulations} 
 we demonstrate how the rates of convergence depend on 
 the sample size. The statistical method is particularly useful 
 for the comparison of compound Poisson (finite intensity) processes 
 and for identifying power law distributions in time series, 
 as will be demonstrated using meteorological and climatological examples in Section \ref{sec: climate examples}. 
 The first set of time series considered is a the well-studied example of the 
 Greeland Ice Core Project (GRIP) paleoclimatic proxies 
 indicating abrupt transitions between stadials and interstadials 
 of the last glacial period. 
 In \cite{Dit99a} Ditlevsen identified 
 an $\alpha$-stable L\'evy component in the GRIP calcium concentration record 
 through consideration of the statistics of time series increments. 
 Further investigation in that direction has been carried out 
 in \cite{GI14} and \cite{HIP09} exploiting the scaling properties of these processes. 
 As a second example we consider long records of precipitable water vapor 
 in the Western Tropical Pacific (Manus and Nauru) 
 which are governed by dynamics with a strong threshold effect and exhibit clearly 
 recognizable heavy-tailed patterns. 
 Our approach is more general than those used previously. 
 In particular, it does not rely on scaling properties of the time series. 
 A preliminary investigation of heavy tail behavior in time series in \cite{GHKK16}  
 is systematically extended in this article, resulting in a robust and well-tested 
 tool for the detection of polynomial tails in time series. 
 
 The article is organized as follows. 
 In Section \ref{sec: Levy and coupling} we introduce L\'evy processes 
 and motivate the notion of Wasserstein distance between probability measures 
 and coupling distance between L\'evy measures before stating the main theoretical estimate proved in \cite{GHKK14}. 
 In Section \ref{sec: stat and theory} we describe the implementation of the coupling distance 
 and derive asymptotic rates of convergence. 
 Section \ref{sec: simulations} considers the application of the method to synthetic time series, and includes 
 sensitivity analyses related to the data size and model parameters. 
 In Section \ref{sec: climate examples} 
 we analyze climatic time series to assess evidence of heavy tailed behavior.

\bigskip
\section{Coupling distances between L\'evy measures}\label{sec: Levy and coupling}

\subsection{L\'evy processes } \label{subsec: Levy}
L\'evy processes are the common generalization of Poisson processes 
and Brownian motion. 
Both processes have statistically independent and stationary increments.  
In the case of a Poisson process $P = (P(t))_{t\gqq 0}$, 
(with intensity $\la>0$) given by $P(t) - P(s) \sim P(t-s) \sim \mbox{Poi}(\la(t-s))$, 
and in the case of a (standard, scalar) Brownian motion $B = (B(t))_{t\gqq 0}$ with 
$B(t) - B(s) \sim B(t-s) \sim \nN(0, t-s)$. 
In addition, each of these processes starts at $0$ and 
has at least right-continuous paths, 
which enjoy left limits at each point in time $t\gqq 0$. 
The concept of a L\'evy process lifts the 
assumption of an explicitly prescribed stationary increment distribution.  
A L\'evy process $L = (L(t))_{t\gqq 0}$ 
is a stochastic process starting at $0$ 
with independent and stationary increments in the sense that 
$L(t) - L(s) \sim L(t-s)$ and that has trajectories 
which are right-continuous $L(t+) = L(t)$ 
and have left limits in each point of time. 
That is, $L(t-)$ exists, which yields 
a well-defined jump size $\Delta_t L = L(t) - L(t-)$. 
This type of path yields for each point in time $t$ 
the decomposition of the current state 
$L(t) = L(t-) + \Delta_t L$ into the ``predictable'' (that is, continuously 
accessible) part $L(t-)$ of $L(t)$ 
(w.r.t. the underlying flow of information)  
and the ``non-anticipating'' jump increment $\Delta_t L$ of $L(t)$ 
(w.r.t. the underlying flow of information). 
The jump increment is statistically independent from $L(s)$ for all $s<t$ 
and hence also of $L(t-)$ 
due to the independence of increments.
L\'evy processes are well-studied objects 
in the mathematics literature. 

The so-called L\'evy-It\^o decomposition states that 
any given L\'evy process $L = (L(t))_{t\gqq 0}$ in $\RR^k$
can be decomposed path-wise 
into the sum of the following four unique and independent components 
\begin{align}\label{eq: Levy-Ito}
L(t) = a t + A^{1/2} B(t) + C(t)^\rho + J^\rho(t) \qquad~\mbox{ for all }~ t\gqq 0 \quad \PP-\mbox{a.s.}
\end{align}
The vector $a \in \RR^k$ accounts for the linear drift,  
$B = (B(t))_{t\gqq 0}$ is a $d$-dimensional standard Brownian motion, which comes 
with a non-negative definite, symmetric covariance matrix $A \in \RR^{d\otimes d}$. 
The processes $(C^\rho(t))_{t\gqq 0}$ and $(J^\rho(t))_{t\gqq 0}$ are pure jump processes 
which share as a common parameter the so-called L\'evy measure $\nu$ and a threshold $\rho>0$.
The L\'evy measure $\nu$ is given as a sigma-finite measure on the Borel sets 
$\nu: \fB(\RR^k) \ra [0, \infty]$ satisfying for any $\rho>0$ that 
\begin{equation}\label{def: Levy mass}
\nu(\{0\}) = 0, \qquad \int_{\bB_\rho(0)} |x|^2 \nu(dx) < \infty \qquad 
\mbox{ and }\qquad \nu(\RR^k\setminus \bB_\rho(0)) < \infty,  
\end{equation}
where $\bB_\rho(0) = \{y\in \RR^k~|~|y|< \rho\}$. 
For each fixed $\rho>0$ the measure $\nu$ parametrizes two associated processes in the following way.  
The first process, $C^\rho = (C^\rho(t))_{t\gqq 0}$, is a compound Poisson process  
with Poisson jump intensity $\la_\rho = \nu(\RR^k \setminus \bB_\rho(0))$ 
and jump distribution $\nu_\rho$ 
\begin{equation}\label{eq: cpp measure}
\fB(\RR^k) \ni E \mapsto \nu_\rho(E) :=\frac{\nu(E \cap (\RR^k \setminus \bB_{\rho}(0)))}{\la_\rho}.
\end{equation}
That is, $C^\rho$ is a Markovian pure jump process with memoryless and hence exponentially distributed waiting times between 
consecutive jumps of intensity $\la_\rho$, 
which are independent and distributed according to $\nu_\rho$ in (\ref{eq: cpp measure}). 
Note that the jumps of $C^\rho$ are bounded away from zero by $\rho$. 
The second process is another Markovian pure jump process $J^\rho = (J^\rho(t))_{t\gqq 0}$, 
whose jumps are bounded from above by $\rho$ with possibly infinite intensity.
The process $(J^\rho(t))_{t\gqq 0}$ can be understood as follows. 
The Fourier transform of the marginal law $L_t$ 
is given by the so-called L\'evy-Chinchine decomposition 
\begin{align*}
& \EE[\exp(i \lv u, L(t)\rv)] = \exp(t \Psi(u)), \qquad t\gqq 0, u\in\RR^k,
\\[2mm]
&\mbox{ with }\\
& \Psi(u) = i\lv u, a\rv -\frac{1}{2}\lv a, Aa\rv + 
\int_{\RR^k\setminus \bB_{\rho}(0)} \big[e^{i\lv u, y\rv} -1\big] \nu(dy)
+ \int_{\bB_{\rho}(0)} \big[e^{i\lv u, y\rv} -1 - i\lv u, y\rv \big] \nu(dy), \qquad u\in \RR^k.
\end{align*}
This representation tells us that $J^\rho = (J^\rho(t))_{t\gqq 0}$ can be understood as the superposition 
of independent compensated (i.e. re-centered) compound Poisson processes $J^\rho(t) = \sum_{j=1}^\infty (\ti C^j(t) - t \int_{R_j} y \nu(dy))$,  
where the compound Poisson processes $\ti C^j$ take jumps with values in rings given 
by \mbox{$R_j = \{\rho_j< |y|\lqq \rho_{j-1}\}$,} distributed as $E\mapsto \frac{\nu(E \cap R_j)}{\nu(R_j)}$ with intensity $\nu(R_j)<\infty$.  
The sequence $\rho = \rho_0  > \rho_1 > \dots > 0$ of radii is strictly decreasing and  $\rho_j \searrow 0$  
such that $\bB_{\rho}(0)\setminus \{0\} = \bigcup_{j\in \NN_0} R_j$ 
with joint (possibly infinite) intensity
\[
\sum_{j\in \NN_0} \nu(R_j) = \nu(\bB_{\rho}(0)) \lqq \infty.
\]
Note that if the total jump intensity $\nu(\bB_{\rho}(0))< \infty$ we can choose $\rho =0$ and hence $J^\rho \equiv 0$. 

\paragraph{The case of $\alpha$-stable processes: } 
The physically most familiar class of L\'evy processes beyond the Poisson process and Brownian motion 
are the so-called stable processes, sometimes 
also referred to as L\'evy flights. 
They have a parameter $\al\in (0, 2)$ and enjoy the following 
scaling property 
\begin{align*}
\frac{1}{c^\frac{1}{\al}} L(c t) \stackrel{d}{=} L(t), \qquad \mbox{ for all }c>0 \mbox{ and }t\gqq 0.
\end{align*}
$\alpha$-stable processes 
are pure jump processes of infinite intensity 
with the heavy-tailed L\'evy measure 
\begin{align}\label{eq: stable Levy measure}
\nu(dz) = \frac{c_-}{|z|^{1+\alpha}} \ind\{z<0\} +  \frac{c_+}{z^{1+\alpha}} \ind\{z>0\}.
\end{align}
such that $c_+ + c_->0$. We make the following remarks on basic properties of $\alpha$-stable processes: 
\begin{enumerate}
 \item The situation that $c_- \neq c_+$ implies that 
the L\'evy measure is asymmetric 
and also produces asymmetric $\alpha$-stable marginal laws, 
which however turn out to be still self-similar. 
For more on this property we refer to Sato \cite{Sato-99}. 
 \item The marginal distributions have all smooth densities. 
Closed forms of the densities however are only known in special cases. 
For instance the $(\alpha=1)$-stable process is known to be necessarily 
symmetric $c = c_- = c_+$ and is given by the standard Cauchy process. 
 \item All $\alpha$-stable processes are pure jump processes and 
only allow finite moments of order $p<\al<2$, as is evident from 
the equivalent integrability condition $\int_{|z|>1} |z|^p \nu(dz) < \infty$. 
 \item The degenerate case of a strictly asymmetric L\'evy measure (\ref{eq: stable Levy measure}) given for instance 
for $c_- = 0$ and $c_+ =1$ yields a jump process with jumps only in positive direction. 
In this case the $\alpha = \frac{1}{2}$-stable L\'evy process is called the L\'evy subordinator.  
It can be constructed as the random clock given as the points in time 
that a Brownian motion $B$ catches up with the linear function $t\mapsto t/\sqrt{2}$ 
\[
t \mapsto T_t := \inf\{s>0~|~B(s)> \frac{t}{\sqrt{2}}\}. 
\]
For further details on the distribution and density on this relation we refer to \cite{Sato-99}. 
 \item The limiting case of $\alpha = 2$ is also 
necessarily symmetric and corresponds to a Brownian motion. 
However Brownian motion enjoys very different properties 
than stable processes for parameter $0< \al < 2$.  
For instance, it has continuous sample paths and exponential moments. 

\end{enumerate}

\subsection{Coupling distance} 
\paragraph{The standard case of the coupling distance between L\'evy measures: } 

In the article \cite{GHKK14} the authors construct a special metric, 
the so-called coupling distance,  
on the set of L\'evy measures $\nu$ in $\RR^k$. 
This metric exploits the fact that for any $\rho>0$ the tail-measure 
$\nu_\rho$ as defined before is a probability measure. 
Hence, having two L\'evy processes $L$ and $L'$ with 
L\'evy measures $\nu$ and $\nu'$, the basic idea is now  
to find cutoffs $\rho$ and $\rho'$ such that the intensities $\la_\rho$ and $\la'_{\rho'}$ are 
equal, that is the $C^\rho$ and $C'^{\rho'}$ have the same intensity.  
Then, to compare the jump measures 
it is natural to compare the probability measures $\nu_\rho$ and $\nu'_\rho$, 
which we will denote by $\mu$ and $\mu'$ for short. 

\paragraph{The Wasserstein distance: } 
In our case the comparison of two given probability distributions $\mu$ and $\mu'$ relies on the 
so-called Wasserstein distance between $\mu$ and $\mu'$ on $\RR^k$.  
For the construction of the Wasserstein distance between $\mu$ and $\mu'$ 
we consider all joint distributions $\Pi$ on the product space $\RR^k \times \RR^k$ 
having these two marginal distributions, that is 
$\Pi(E\times \RR^k) = \mu(E)$ and $\Pi(\RR^k\times E) = \mu'(E)$ for any $E\in \bB(\RR^k)$. 
If $\mu = \mu'$, then the special joint distribution $\Pi(E\times E') = \mu(E\cap E')$ concentrates all the mass 
on the diagonal $\{(x, x)~|~x\in \RR^k\}$ of the product space. 
In particular the mean-square of the distance from the diagonal 
$\int_{\RR^k\times \RR^k} |x-y|^2 \Pi(dx,dy) = 0$. 
If $\mu \neq \mu'$ there is obviously not such a joint distribution, that is 
$\int_{\RR^k\times \RR^k} |x-y|^2 \Pi(dx,dy) >0$. 
However, it turns out to be a good measure of proximity of $\mu$ and $\mu'$ 
to minimize over all joint distributions or \textit{couplings} of $\mu$ and $\mu'$ given by 
\begin{align}
\fC(\mu, \mu') &:= \Big\{\Pi: \fB(\RR^k) \otimes \fB(\RR^k) \ra [0, 1] \mbox{ probability measure, with  }\nonumber\\
&\qquad ~\Pi(E \times \RR^k) = \mu(E), \quad \Pi(\RR^k \times E) = \mu'(E) \quad \mbox{ for all } E \in \fB(\RR^k)\Big\}.
\label{eq: def coupling}
\end{align}
The Wasserstein metric of order~$2$ between two probability measures $\mu, \mu'$ 
on the Borel sets of $\RR^k$ is defined by 
\begin{align*}
W_{2}(\mu, \mu') := \inf_{\Pi \in \fC(\mu, \mu')} \Big(\int_{\RR^k \times \RR^k} |x-y|^2 \Pi(dx, dy)~\Big)^\frac{1}{2}.
\end{align*}
Any minimizer $\Pi^*$ is referred to as \textit{optimal coupling} between $\mu$ and $\mu'$. 
It is obvious that in terms of random vectors $(X, X')$ with $X\sim \mu$ and $X'\sim \mu'$ and joint distribution $\Pi$
the Wasserstein-distance has the following equivalent form 
\begin{align*}
W_{2}(\mu, \mu') := \inf_{\substack{\Pi \in \fC(\mu, \mu')\\ (X, X')\sim\Pi }} \EE_\Pi\big[|X-X'|^2\big]^{\frac{1}{2}}.
\end{align*}

\paragraph{Properties of the Wasserstein distance: } 
It is well-known in the mathematics literature \cite{RR1998} that 
the convergence $W_{2}(\mu_n, \mu) \ra 0$ is equivalent to 
the weak convergence of the laws $\mu_n\rightharpoonup \mu$, 
that is, the convergence $\mu_n(B)\lra \mu(B)$ for all ``reasonable'' events $B$ 
and the convergence of the second moments. 
In this work we will consider the following case. 
The optimal coupling between $\mu$ and $\mu'$ is not in general known explicitly, 
with the notable exception of $k=1$, where a parametrization of the optimal coupling $\Pi$ 
is given as follows. 
On $(\RR, |\cdot|)$ one can show that for two distribution functions $F(x) = \mu((-\infty, x])$ 
and $F'(x) = \mu'((-\infty, x])$ the optimal coupling is realized by the random vector 
\begin{equation}\label{eq: optimal coupling}
(X,Y) = (F^{-1}(U), (F')^{-1}(U)),
\end{equation}
where $F^{-1}$ and $(F')^{-1}$ are the quantile functions of $F$ and $F'$ and 
$U$ has the uniform distribution in $[0, 1]$. 
Therefore the Wasserstein metric is easily evaluated by 
\begin{equation}\label{eq: optimal Wasserstein}
W_2^2(\mu, \mu') = \int_0^1 |(F^{-1}(u)- (F')^{-1}(u)|^2 du.  
\end{equation}
It is worth noting that the optimality of the pair $(X, Y)$ 
holds only with respect to the specific (spatial) metric $d(x, y) = |x-y|$ on $\RR$. 
The law of $(X, Y)$ is obviously a coupling of $\mu$ and $\mu'$.  
For metrics other than $d$ the right-hand side of (\ref{eq: optimal Wasserstein}) 
is in general an upper bound. 
If we consider the Wasserstein distance 
for other metric, such as 
the cut-off metric $d_s(x, y) = \min(|x-y|\wedge s)$ for some $s>0$ on $\RR$, 
we lose the result that the specific coupling (\ref{eq: optimal coupling}) is minimal. 
However, since the Wasserstein distance is given as the infimum over all couplings 
with respect to $d$ the means square distance 
with (\ref{eq: optimal coupling}) yields at least an upper bound. 
The real line with the cut-off metric $d_s$ will be the space where we approximate the laws of the jumps, 
since we are considering processes which may not have second moments. 
For further results we refer to \cite{GHKK14} and \cite{GHKK16}.

Note that Wasserstein distances can be considered on much more general spaces than $\RR^k$. 
For instance consider two stochastic processes $(X̣^j_t)_{t\in [0, 1]}$ with continuous trajectories, 
so $\mbox{Law}(X^j)$ is a distribution on $(\cC([0, 1]), \bB(\cC([0,1]))$. Then 
\[
W_2(\mu, \mu') = \inf_{\Pi \in \fC(\mu, \mu')} \Big(\int_{\cC([0, 1]) \times \cC([0, 1])} d(x, y)^2 \Pi(dx, dy)~\Big)^\frac{1}{2}
\] 
is well-defined in the usual sense, where $d(x, y)$ is any complete metric in $\cC([0,1])$, 
for instance $d(x,y) =\|x-y\|_\infty =\sup_{t\in [0, 1]} |x(t)-y(t)|$ or $d(x,y) = \|x-y\|_\infty \wedge s = \sup_{t\in [0, 1]} d_s(x(t), y(t))$. 
The same is true for the more general space $\DD([0, 1])$ of right-continuous functions with left limits $v: [0, 1] \ra \RR$, that is satisfying 
$v(t+) = v(t)$ and $v(t-) \in \RR$ for any $t\in [0, 1]$.  
This is the natural space in which L\'evy processes in the sense of (\ref{eq: Levy-Ito}) 
have their paths. Analogously to the case of continuous functions, the Wasserstein distance is well-defined. 
For instance, for two L\'evy processes $L$ and $L'$ the Wasserstein distance 
between their laws $\mu = \mbox{Law}(L)$ and $\mu' = \mbox{Law}(L')$ in $\DD([0,1])$ is given by 
\begin{equation}\label{eq: WasserOnSkorohod}
W_{2, d}(\mu, \mu') = \inf_{\Pi \in \fC(\mu, \mu')} \Big(\int_{\DD([0, 1]) \times \DD([0, 1])} d(x, y)^2 \Pi(dx, dy)~\Big)^\frac{1}{2}
\end{equation}
using for identical metrics $d$ as before.

\paragraph{The definition of the coupling semimetric and the coupling distance, standard case: } 
We are now in the position to define  
the coupling semimetric family and the coupling distance between two L\'evy measures $\nu$ and $\nu'$. 
\begin{defn}\label{def: coupling distance}
Given two absolutely continuous L\'evy measures 
$\nu = f dx$ and $\nu'= f' dx$ on $\RR^k$ and 
a given intensity $0< \la < \min(\nu(\RR^k), \nu'(\RR^k))$ we define the 
cutoffs
\[
\rho(\la) := \inf\{r>0~|~ \nu(\RR^k\setminus \bB_r(0)) \lqq \la\}
\]
and $\rho'(\la)$ analogously.
We then introduce a family of semi-metrics $T_\la$ 
\[
T_\la(\nu, \nu') := \la^{\frac{1}{2}}\; W_{2}(\nu_{\rho(\la)}, \nu'_{\rho'(\la)}).
\]
\end{defn}
First note that the absolute continuity yields that 
given $\nu$ and intensity $\la>0$ 
we have that the mass of the L\'evy measure outside $\rho(\la)$ 
is just $\la$, that is formally $\la_{\rho(\la)} = \la$, 
with $\la_\rho$ defined before (\ref{eq: cpp measure}). 
Intuitively, given L\'evy measures $\mu$ and $\mu'$, 
the semi-metric $T_\la(\nu, \nu')$ compares the jump distributions 
$\nu_{\rho(\la)}$ and $\nu'_{\rho'(\la)}$  of two compound Poisson processes with common intensity  
$\la= \la_{\rho(\la)} = \la_{\rho'(\la)}$. 

A word about the renormalization pre-factor $\la^{\frac{1}{2}}$:  
roughly speaking, it comes from the fact that an increasing intensity $\la$ 
makes the cutoffs $\rho$ approach $0$. 
However around $0$ the L\'evy measure has a well-known ``worst-case'' pole 
associated with the integrability condition $\int_{-1}^1 z^2 \nu(dz) <\infty$ 
of the L\'evy-Chinchine decomposition, which has to be satisfied for any L\'evy measure. 
The pre-factor $\la^{1/2}$ allows the coupling distance to remain finite even for large $\la$. 
We illustrate why this is the case using the one-sided L\'evy measure  
$\nu(dz) = \frac{dz}{|z|^{1+\alpha}} \ind_{(0, 1)}(z)$, 
$z\in (0,1)$ for some $\alpha\in (0,2)$. 
\footnote{Strictly, one actually has to consider the two sided case, 
since the L\'evy-Chinchine decomposition for one-sided case 
is slightly more restrictive, see for instance Applebaum \cite{Ap09}.}
In this case we have $\nu(0, 1) = \infty$. 
Consider now two L\'evy measures $\nu, \nu'$ of this type with 
exponents $\al, \al'\in (0,2)$. 
Calculating their optimal coupling explicitly we observe the following. 	
For given intensity $\la\gg 1$ we calculate the cutoff $\rho$ of $\nu$ as
\begin{align*}
\la = 2 \int_\rho^1 \frac{dz}{z^{1+\al}} = \frac{2}{\al}\Big(\frac{1}{\rho^\al}-1\Big) \quad \vzv \quad 
\rho = \Big(\frac{2}{2+\al \la}\Big)^{\frac{1}{\al}}.	
\end{align*}
Hence the tail distribution function and the quantile function satisfy 
\begin{align*}
&F_\rho(x) = \nu_\rho([x, 1)) = \frac{1}{\la} \int_x^1 \frac{dz}{z^{1+\al}} = \frac{1}{\la} \frac{2}{\al}\Big(\frac{1}{x^\al}-1\Big) 
\mbox{ for }~\rho \lqq x <1 \qquad \mbox{ and hence} \\
&F_\rho^{-1}(y) = \Big(\frac{2}{\la\al y +2}\Big)^\frac{1}{\al} = \la^{-\frac{1}{\al}} \Big(\frac{\frac{2}{\al}}{y +\frac{2}{\al}\frac{1}{\la}}\Big)^\frac{1}{\al}. 
\end{align*}
Analogously we obtain ${\nu'}_{\rho'}$ and ${F'}_{\rho'}^{-1}(y)$ and calculate with the help 
of the triangular inequality for $L^2(0,1)$ and a standard expansion for large $\la\gg 1$ 
\begin{align*}
\la^\frac{1}{2} \wW_{2}(\nu_\rho, {\nu'}_{\rho'}) 
&= \Big(\int_0^1 \la  (F_\rho^{-1}(y)- {F'}_{\rho'}^{-1}(y))^2 dy \Big)^\frac{1}{2} \\
&= \Big(\int_0^1 \Big[\la^{\frac{1}{2}-\frac{1}{\al}} \Big(\frac{\frac{2}{\al}}{ y +\frac{2}{\al}\frac{1}{\la}}\Big)^\frac{1}{\al}
- \la^{\frac{1}{2}-\frac{1}{\al'}} 
\Big(\frac{\frac{2}{\al'}}{y +\frac{2}{\al'} \frac{1}{\la}}\Big)^\frac{1}{\al'}\Big]^2 dy \Big)^\frac{1}{2}\\
&\lqq \la^{\frac{1}{2}-\frac{1}{\al}} \Big(\int_0^1 \Big(\frac{\frac{2}{\al}}{y +\frac{2}{\al}\frac{1}{\la}}\Big)^\frac{2}{\al} dy\Big)^\frac{1}{2} 
+ \la^{\frac{1}{2}-\frac{1}{\al'}}\Big(\int_0^1 \Big(\frac{\frac{2}{\al'}}{y +\frac{2}{\al'}\frac{1}{\la}}\Big)^\frac{2}{\al'} dy\Big)^\frac{1}{2}\\
&\stackrel{\la \gg 1}{\approx} \la^{\frac{1}{2}-\frac{1}{\al}} \la^{1-\frac{2}{\al}} + \la^{\frac{1}{2}-\frac{1}{\al'}} \la^{1-\frac{2}{\al'}} 
= \la^{3(\frac{1}{2}-\frac{1}{\al})}  + \la^{3(\frac{1}{2}-\frac{1}{\al'})}.
\end{align*}
Since $\al, \al' < 2$ we obtain that even for this worst case pole 
around $0$ there the function $\la \mapsto \la^\frac{1}{2} \wW_{2}(\nu_\rho, {\nu'}_{\rho'})$ 
is asymptotically decreasing, and hence globally bounded. 
Taking the supremum over all $\la>0$ is hence a meaningful operation.

Note further that the bivariate function $T_\la$ is symmetric and satisfies the triangle inequality hence 
it is a semi-metric. 
Clearly, \mbox{$T_\la(\nu, \nu') = 0$} does not guarantee that $\nu = \nu'$, 
since the values of $\nu|_{B_{\rho(\la)}}$ are not taken into account. 
Therefore it is not a proper metric despite being symmetric and satisfying the triangular inequality. 
For details we refer \cite{GHKK14}. 

To overcome the shortcomings of using a semi-metric we take the following approach. 
We assume that the L\'evy measure has a density with respect to the Lebesgue measure 
and that it has infinite mass $\nu(\RR^k) = \infty$. Hence, the 
cut-off as a function of the intensity $\la \mapsto \rho(\la)$ is 
continuous and monotonically decreasing with $\lim_{\la \ra \infty} \rho(\la) = 0$: 
the set of increments not taken into account by the semi-metric $T_\la$ 
given by $B_{\rho(\la)}(0)$ is strictly decreasing. Therefore it makes sense 
to take as a distance between two L\'evy measures the supremum $T_\la$ over all $\la>0$. 

\begin{defn}\label{def: td 1}
For two absolutely continuous L\'evy measures 
$\nu = f dx$ and $\nu'= f' dx$ on $\RR^k$ with $\nu(\RR^k) = \nu'(\RR^k) = \infty$ 
we define the \textbf{coupling distance} 
\[
T(\nu, \nu') := \sup_{\la>0} T_\la(\nu, \nu'). 
\]
\end{defn}
We stress that both the restrictions of absolutely continuous measures and infinite activity 
in the definition above can be removed. For details we refer to the original work \cite{GHKK14}. 
On the one hand the restriction on absolute continuity is overcome by an interpolation procedure. 
The requirement of infinite mass can be dropped by the ad hoc introduction of an 
artificial point mass in $0$ carrying the missing weight. 
In this way, finite L\'evy measures are accommodated in this framework. 
We also refer to \cite{GHKK16},  for further applications of coupling distances. 

\paragraph{Coupling distances between L\'evy measures, the general case: } 

The fact that a finite Wasserstein distance and hence 
finite coupling distance require finite second moments implies that 
coupling distances between $\alpha$-stable L\'evy measures 
are not yet covered by Definition \ref{def: td 1}.
In order to treat the distance between $\alpha$-stable L\'evy SDEs, 
we slightly modify the definition of $T_\la$ and $T$ with the help 
of the truncated Wasserstein distance. 
For $s>0$ given we define 
\[
\ti W_{2, s}^2(\mu, \mu') = \inf_{\Pi \in \fC(\mu, \mu')} \Big(\int_{\RR\times \RR} (|x-y|^2\wedge s) ~\Pi(dx, dy)~\Big)^\frac{1}{2}
\]
and analogously 
\[
\ti T_\la(\nu, \nu') = \la^\frac{1}{2} \ti W_{2, s}^2(\nu_{\rho(\la)}, \nu'_{\rho'(\la)}), \qquad \ti T(\nu, \nu') 
= \sup_{\la >0} T_\la(\nu, \nu').
\]
This truncation allows us to measure the proximity of L\'evy measures. 
While small values $|x - y| \lqq \sqrt{s}$ are unaffected, 
large values of $|x - y|$, which violate  the square integrability 
are set equal to a constant. As a result, $\ti T_\la(\nu, \nu')\lqq (\la s)^\frac{1}{2}$. 
Clearly, the coupling (\ref{eq: optimal coupling}) is no longer the optimal coupling 
for the modified Wasserstein distance $\ti W_2^2$, which cannot be  
evaluated analogously to (\ref{eq: optimal Wasserstein}). 
However, for the following truncated version of (\ref{eq: optimal Wasserstein})
\begin{equation}\label{eq: suboptimal Wasserstein}
\ti W_{2, s}^2(\mu, \mu') \lqq \int_0^1 \Big( |(F^{-1}(u)- (F')^{-1}(u)|^2\wedge s\Big) du,  
\end{equation}
the Wasserstein distance given as \textit{minimal} coupling of this cutoff distance is less than 
equal to the \textit{given} coupling on the right hand side of equation \ref{eq: optimal Wasserstein}. 

\subsection{The central estimates } 

We now turn to the estimates of the Wasserstein distance of 
the laws of two solutions of stochastic differential equations 
driven by different L\'evy noise processes on path space, 
making use of standard distances on the parameters 
including the coupling distance between the 
underlying L\'evy measures. 
Note that from now on we will work throughout the text with non-dimensionalized quantities. 
This is justified later on  in Section \ref{sec: simulations} 
by the scaling property of our estimator shown in Figure \ref{fig:scaling}. 
In principal the scale used for non-dimensionalization is arbitrary. 
In our analysis, we will focus on the interquartile range as this is defined for all random variables.

We consider two stochastic differential equations $j=1,2$ of the following type 
\begin{align}
\label{eq1}
X_j(t)=x_j+\int_{0}^{t}f\left(X_j(s)\right)ds+L_j(t) \qquad t\gqq 0, ~x_j\in \RR.
 \end{align}
where each $L_j$ is a L\'evy process in the sense of (\ref{eq: Levy-Ito}) with the characteristic triplet $(a_j, A_j, \nu_j)$ and 
the vector field $f: \RR\ra \RR$ is continuous and satisfies the following monotonicity condition 
\begin{align}\label{one-Lip}
(f(x)-f(y))(x-y)\lqq \ell (x-y)^2, \quad x, y\in \RR, \mbox{ for some }\ell >0. 
\end{align}
It is well-known that unique strong Markovian solutions to (\ref{eq1}) 
are obtained analogously to ordinary differential equations 
via a Picard - type successive iteration procedure.

Generic examples satisfying (\ref{one-Lip})
are all globally Lipschitz continuous functions 
and polynomials of odd order with negative leading coefficient
$$
f(x)= b_nx^n + \sum_{i=0}^{n-1} b_i x^i, \qquad b_n<0, ~b_i \in \RR, ~n \mbox{ odd}.
$$
In applications such a polynomial is often considered
as the gradient $f=-U'$ of an energy potential $U: \RR\ra \RR$ with 
several local minima.
In climate science a class of such examples examples can be derived 
from radiative energy balance considerations (\cite{BenziEtal}, \cite{Im01}, \cite{IM02}, \cite{Dit99a} and references therein).
Under the proper choice of parameters the potential $U$ admits two local minima that correspond to two climate equilibrium states 
(we refer to the discussion in Section \ref{sec: climate examples}.1). 
In neuroscience the FitzHugh-Nagumo model is an example of such a fourth order potential $U$.
In this case equation (\ref{eq1}) models the membrane voltage under random excitation,
which has been studied recently in the literature (e.g. \cite{BerLan}, \cite{DoThieu}, \cite{TucRodWan}).

The first of our main theoretical results estimates the deviation
between the laws of the solutions $X_j$ on $(\DD(0,1), \ti d)$
in terms of the metric induced by $T$ and $T_\la$ on the set
of tupel of parameters
$(x_j, a_j, A_j, \nu_j)$. 

\begin{thm}\label{Alex-theor-2}
Let $f: \RR \ra \RR$ be $\cC^2$ and satisfy condition (\ref{one-Lip}) for some constant $\ell>0$ 
and $(a_j, A_j, \nu_j)$ be two L\'evy characteristics and given initial values $x_j \in \RR$, $j=1,2$. 

Then for any two solutions $X_j$ of equation (\ref{eq1}) driven by L\'evy processes $L_j$
with the respective characteristics there are uniform constants $c_1, c_2>0$ such that
\begin{equation}\label{final_est-2}\begin{aligned}
\wW_{2, \ti d}^2\Big(\mathrm{Law}(X_1), \mathrm{Law}(X_2)\Big)\lqq c_1\;Q^1 e^{\ell/\arctan(1/2)} + c_2\;Q^2,
\end{aligned}
\end{equation}
where $\wW_{2, d}$ is the Wasserstein distance on the 
path space $\DD([0, 1])$ equipped with the cutoff metric $\ti d(g,h) := \|g-h\|_\infty \wedge s$ 
in the sense of (\ref{eq: WasserOnSkorohod}) and
\begin{align*}
Q^1&= |x_1 - x_2|^2\wedge s + |a_1-a_2|+(\sqrt{A_1}-\sqrt{A_2}\,)^2+T^2(\nu_1, \nu_2)\nonumber\\[2mm]
& \quad +\big(\nu_1(|u|>s)+\nu_2(|u|>s)\big)^{1/2} T(\nu_1, \nu_2),\\[3mm]
Q^2 &= \sqrt{(\sqrt{A_1}-\sqrt{A_2})^2+T^{2}(\nu_1, \nu_2)}.
\end{align*}
\end{thm}

This result states that the Wasserstein distance of the optimal paths 
in the function space $\DD([0, 1])$ (that is for all $t\in [0, 1]$ simultaneously) 
is estimated by the distances of the characteristics of the marginals $L^j_t$, 
which are identical for all $t\in [0, 1]$. In addition, the Wasserstein distances 
on the right-hand side are between L\'evy measures in $\RR$, such that the optimal coupling 
is either known (in the case without cutoff $s= \infty$) or can be estimated in the case $s<\infty$. 
This result is the strongest possible since 
it compares the mean-square of the paths uniformly on a given time interval. 
The proof of this result is a consequence of the following more technical but 
practically much more useful result, which includes the choice of a level of intensity $\la$ 
and the estimation of $\wW_{2,\ti d}(\mathrm{Law}(X_1), \mathrm{Law}(X_2))$ by a respective semi-metric $T_\la$. 
In addition, all constants are given explicitly. 
This result is primarily of theoretical value 
as it is impossible in practice 
to evaluate the coupling distance 
simultaneously for all intensities $\la>0$, 
since this requires an arbitrarily large number of data points.
Therefore we state a simpler version for fixed finite intensity $\la>0$. 
Since we work with non-dimensionalized data it is enough to state 
the following result with $s=1$. 

\begin{thm}\label{Alex-theor}
Let $f: \RR \ra \RR$ be $\cC^2$ and satisfy condition (\ref{one-Lip}) for some constant $\ell>0$ 
and $(a_j, A_j, \nu_j)$ be two L\'evy characteristics and given initial values $x_j \in \RR$, $j=1,2$. 
In addition, let the expression $\wW_{2, \ti d}$ be defined as in Theorem \ref{Alex-theor-2}.

Then for any $\la>0$ and any two solutions $X_j$ of equation (\ref{eq1}) driven by L\'evy processes $L_j$
with the respective characteristics the following estimate holds true
\begin{align*}\label{final_est}
\wW_{2, \ti d}^2\Big(\mathrm{Law}(X_1), \mathrm{Law}(X_2)\Big)
&\lqq Q_{\la}^1 e^{\ell/ C_0 } + Q_\la^2,
\end{align*}
where
\begin{align*}
Q_\la^1&= 2 (|x_1- x_2|^2\wedge 1) + C_1 \Big(C_2|a_1-a_2|+(\sqrt{A_1}-\sqrt{A_2}\,)^2+U_{\la}
\left(\nu_{1}\right)+U_{\la} \left(\nu_{2}\right)+\nonumber\\[2mm]
& \quad + C_3 T^2_\la(\nu_1, \nu_2) +C_4 \min(\nu_1(|u|>1) + \nu_2(|u|>1), \la)^{1/2} T_\la(\nu_1, \nu_2)\Big),\\[3mm]
Q_\la^2 &= C_1 \sqrt{ C_5 (\sqrt{A_1}-\sqrt{A_2})^2
+C_6 (U_{\la}\left(\nu_{1}\right)+U_{\rho}\left(\nu_{2}\right)+T^{2}_{\la}(\nu_1, \nu_2))},\\[3mm]
& U_{\la}\left(\nu_{j}\right)=\int_{|u|\lqq \rho_{j}(\la)}u^2\nu_{j}(du),\qquad j=1,2,
\end{align*}
with the following numerical values. 
\begin{center}
\begin{tabular}{|l|l|l|l|l|l|l|l|}
\hline
Constant    & $C_0$          & $C_1$           & $C_2$               & $C_3$         & $C_4$             & $C_5$     & $C_6$ \\
\hline
Exact value & $\arctan(1/2)$ & $ 4 / \pi$ & $ 3^{3/4}/ 2$ & $\pi+3^{3/4}$ & $(\pi+3^{3/4})/2$ & $3^{3/2}$ & $(2\pi)^2$\\
\hline
Numerical approx. & $0.46$       & $1.27$        & $1.40$            & $5.42$      & $2.28$          & $5.20$  & $39.48$ \\
\hline
\end{tabular}
\end{center}
\end{thm}

\bigskip
\section{Statistical implementation and rates of convergence} \label{sec: stat and theory}

In this section we develop a practical implementation of the program laid out in the previous sections. 
In order to keep both calculations and strategy tractable 
we will restrict ourselves to a simple example, 
which can be adapted and extended. 
While this section has some overlap with Section 3 in \cite{GHKK16} 
we present a simplified and streamlined version 
of the proof of the rates of convergence. 
We start with some basic statistics and the implementation 
of the sample Wasserstein distance between the empirical measure 
and reference measures. With these tools at hand we prove 
rates of convergence for the sample distance. 
In the last subsection we explain our procedure in detail.

\subsection{Basic statistics }\label{subsec: 3.1}

In this section we provide the technical background to 
compare the jump statistics of a data set 
to a given reference distribution in terms of the coupling distance. 
Since our focus 
is essentially one dimensional we concentrate on the scalar case. 
For higher dimensions we refer to the Conclusion. 

\subsubsection{Basic notions and results } 

For a given sequence of independent 
and identically distributed random variables $(X_i)_{i\in \NN}$ 
with distribution $\mu$ on $\RR$ we 
denote by $\mu_n$ the empirical distribution based on the 
sample of size $n$ defined as 
\begin{equation}
  \label{eq:mu_n}
  \mu_n(E):=
  \frac{\#\{X_i\in E\}}{n}=\frac{1}{n}\sum_{i=1}^n\ind{\{X_i\in E\}}\ , \quad E\in \fB(\RR) \ .
\end{equation} 
The corresponding empirical distribution function
$F_n$ is the distribution function of $\mu_n$
\begin{equation}
  \label{eq:emp_dist_function}
  F_n(x):=\mu_n((-\infty,x])=\frac{\#\{X_i\lqq x\}}{n}=\frac{1}{n}\sum_{i=1}^n\ind{\{X_i\lqq x\}}\ , \quad x\in\RR \ .
\end{equation} 
A version of the strong law of large numbers known as the Glivenko-Cantelli Theorem 
tells us that for the distribution function $F(x) = \mu((-\infty, x])$, $x\in \RR$, of $\mu$ 
we have for almost all $\om \in \Omega$ the uniform convergence 
\begin{equation}\label{eq: lln}
\sup_{x\in\RR}|F_n(x, \om)-F(x)|\to 0 \quad \mbox{ as } n\to\infty\ .
\end{equation}

Non-trivial limits of the quantity in (\ref{eq: lln}) are obtained by the following rescaling. 
For any fixed $x\in\RR$, the random variables $\ind\{X_i\lqq x\}$, $i\in\NN$ are i.i.d. 
Bernoulli variables with $F(x) = \PP(X_i\lqq x)$. 
Hence the central limit theorem (de Moivre-Laplace Theorem) states that 
\begin{equation}\label{eq:deMoivreLaplace}
  \sqrt{n}\left(F_n(x)-F(x)\right) = \frac{1}{\sqrt{n}}\sum_{i=1}^n\ind\{X_i\lqq x\} - \sqrt{n}F(x) \stackrel{d}{\lra} \nN\left(0,F(x)(1-F(x))\right).
\end{equation}

\paragraph{Quantiles: } To compute the Wasserstein distance it is necessary to consider the 
empirical quantile function 
 $F_n^{-1}$.
 For $n\in\NN$ denote by $X_{i:n}$ the $i$-th order statistic 
of a sample of size $n$, i.e. $i$-th smallest element of the the ordered sample 
\[X_{1:n}\lqq \dots \lqq X_{i:n}\lqq\cdots\lqq X_{n:n}.\]
The quantile $F_n^{-1}$ of $F_n$ can then be expressed in terms of the order statistics 
since for $ 0<u\lqq1$
\begin{equation}
\label{eq:empirical_quantiles}
\begin{split}
  F_n^{-1}(u)&=\inf\{x\in\RR:\ F_n(x)\gqq u\}=\inf\{x\in\RR: \#\{X_i\lqq x \}\gqq nu\}
  =X_{\lceil nu\rceil:n}  \ .
\end{split}
\end{equation}
In analogy to (\ref{eq: lln}) the strong law of large numbers implies that 
$$
|F^{-1}_n(u, \om)-F^{-1}(u)| \ra 0\ ,  \qquad \mbox{ for all } u\in (0,1) \mbox{ as } n\ra \infty \mbox{ for almost all } \om \in \Omega.
$$
However, in general we cannot expect uniform convergence in $u$, 
since unbounded support of $\mu$ implies that 
the values of the quantile $F^{-1}(u)$ for $u$ tending to $0,1$ will tend to infinity. 
To obtain the analogy to the central limit theorem of formula 
(\ref{eq:deMoivreLaplace}) we introduce the well-known concept 
in time series analysis of empirical quantile process. 

\begin{defn}\label{def: quantile process}
Let $(X_i)_{i\in\NN}$ be a sequence of real valued i.i.d. random variables with common distribution function $F$. 
The empirical quantile process is defined as
  
  \begin{equation}
      Q_n(u,\omega):=\sqrt{n}\left(F^{-1}_n(u,\omega)-F^{-1}(u)\right),\quad 0\lqq u\lqq 1,\ n\in\NN\ .
  \end{equation}
  
\end{defn}

For $\mu$ being the uniform distribution 
we obtain the following analogue of the central limit theorem. 
\begin{thm}
  \label{thm: lim_quantile_proc}
  Let $(X_i)_{i\in\NN}$ be a sequence of real valued i.i.d. random variables with common uniform 
distribution function $F(x) = x$, $x\in [0,1]$. 
Then there exists a Brownian bridge $(B^0_n)_{n\in \NN}$, that is a Brownian motion conditioned to end in $0$ at the endpoint $x=1$, 
such that 
\begin{equation}
\sup_{0\lqq u\lqq 1} |Q_n(u, \om) -B^0_u(\om)|\stackrel{d}{\lra} 0, \mbox{ as } n\rightarrow \infty.
\end{equation}
\end{thm}

As already mentioned, for general distributions $F$ 
uniform convergence of the empirical quantile process is unavailable.  
Instead, we will consider truncated $L^2$ distances later. 
In the following we link the empirical quantile process to 
the approximation of Wasserstein distances.

\subsubsection{The implementation of the empirical Wasserstein distance } \label{subsubsec: empirical Wasserstein distance}

We measure the speed of convergence 
described by the Glivenko-Cantelli theorem in terms of the Wasserstein distance 
as follows. 
\paragraph{The Wasserstein distance between the empirical measure and the reference measure: } 
For some given reference distribution 
$\mu$ with second moments 
we calculate the Wasserstein distance 
between an empirical measure $\mu_n(\om)$ and $\mu$ itself. 
For convenience let us introduce the Wasserstein statistic 
with respect to $d(x, y) = |x-y|$ 
\begin{equation}
\wnq := \wnq(\mu, \om) := W_2^2(\mu_n(\om),\mu), \qquad \mbox{ for }n\in \NN.
\label{eq:wnq}
\end{equation}
This distance can be calculated explicitly due to the explicitly known 
shape of the optimal coupling by  
\begin{equation}
\label{wnq}
  \wnq =\int_0^1|F_n^{-1}(u)-F^{-1}(u)|^2du =\int_0^1|X_{\lceil nu\rceil:n}-F^{-1}(u)|^2du 
  \\
  =\sum_{i=1}^n\int_{\frac{i-1}{n}}^{\frac{i}{n}}(X_{(i-1):n}-F^{-1}(u))^2du.
\end{equation}
The expression on the right side of (\ref{wnq}) turns out to be a 
second order polynomial in the order statistic, that is
\begin{equation}\label{eq: wnp}
  \wnq =\sum_{i=1}^n a_iX_{i:n}^2 + b_iX_{i:n}+c \ ,
\end{equation}
where the coefficients are determined by the binomial formula and given by
\begin{eqnarray}\label{eq: coefficients empirical Wasserstein}
a_i=\frac{1}{n}, \quad b_i=-2\int_{\frac{i-1}{n}}^{\frac{i}{n}}F^{-1}(u)du,
\quad c=\int_{0}^{1}\big(F^{-1}(u)\big)^2du\ .
\end{eqnarray}

\paragraph{The cutoff Wasserstein distance between the empirical measure and the reference measure: } 
The previous result can also be adapted for the cut-off version. 
In this case we consider the Wasserstein distance with respect to $d(x, y) = |x-y|\wedge s$ for 
some cutoff $s>0$. 
For simplicity in this context (and because in practice we will work with non-dimensionalized time series) 
we restrict ourselves 
in the discussion to the case of $s=1$. Other cutoff values are implemented analogously.  
Note that in the cutoff case the optimal coupling is no longer explicit, 
however the analogous measure $|F_n^{-1}(u)-F^{-1}(u)|\wedge 1$ 
remains a (generally suboptimal) coupling. Therefore 
the Wasserstein distance with respect to $d(x, y) = |x-y|\wedge 1$ satisfies 
\begin{equation}\label{def: wns}
W_{2}^2(\mu_n,\mu)\lqq \int_0^1 \big(|F_n^{-1}(u)-F^{-1}(u)|\wedge 1\big)^2 du =: \wns(\mu) =: \wns\ , 
\end{equation}
where the right-hand side can be calculated similarly to (\ref{eq: wnp}) in terms of the order statistics as follows
\begin{equation}\label{eq: wnp cutoff}
\wns =\sum_{i=1}^n A_iX_{i:n}^2 + B_iX_{i:n}+C_i +D.
\end{equation}
For the repartition $0\lqq \ell_1 \lqq r_1\lqq \ell_2 \lqq r_2\lqq \cdots\lqq \ell_n \lqq r_n\lqq 1$ given by
\begin{equation}
\ell_i= \left(\frac{i-1}{n}\vee  |F(X_{i:n}-1)|\right)\wedge \frac{i}{n}  \ , 
\qquad
 r_i= \frac{i-1}{n}\vee\left( |F(X_{i:n}+1)|\wedge\frac{i}{n} \right)
\label{eq:intervals}
\end{equation}
the coefficients are calculated by
\begin{eqnarray}\label{eq: coefficients empirical Wasserstein cutoff}
A_i= r_i-\ell_i, 
\quad B_i=-2\int_{\ell_i}^{r_i}F^{-1}(u) du,
\quad C_i=\int_{\ell_i}^{r_i}\big(F^{-1}(u)\big)^2 du,
\quad D=1-\sum_{i=1}^n A_i\ .
\end{eqnarray}

\subsection{Asymptotic distribution and rate of convergence for power laws}\label{subsec: 3.2} 
For rigorous statistical applications it is necessary 
to determine the rate of convergence of the statistic of interest, in our case $\wns$. 
By definition $\wns$ tends to zero as $n\ra \infty$. Quantifying the rate of convergence 
amounts to finding the correct renormalization to obtain a non-trivial (random) limit. 
For calculational convenience we state the theorem for the one-sided case. 
The negative tail is chosen so that for a given density $f$ the function $f(F^{-1})$ 
is again monotonically increasing. 
The two-sided case is a straight-forward extension 
as the sum of both one-sided coupling distances. 

\begin{thm}\label{thm: rate of convergence}
Assume $f(x) = c^\alpha |x|^{-\alpha-1}$ for all $x<-c$ and some $\alpha>0$. Then for any $\kappa$ with
\begin{equation*}
0<\kappa<\frac{\alpha}{\alpha+2}\ ,
\end{equation*}
there exists a sequence of non-negative random variables $(\mathcal E_{\kappa,n})_{n\in\NN}$ 
(on the same probability space)
such that for all $n\in\NN$
\[
n^{\kappa} \;  \wnq^* \lqq \; \mathcal E_{\kappa,n} \qquad \mbox{$\PP-$almost surely,}
\] 
which converges in distribution 
\[
\mathcal E_{\kappa,n}
\stackrel{d}{\lra}  c^2 \int_0^1 (u^{\frac{1}{\alpha} - 1} B_u)^2 du 
\quad { as }\quad n\ra \infty \ ,
\]
where $B = (B_u)_{u\in [0,1]}$ is a standard Brownian motion. 
\end{thm}

\begin{proof}
First consider a sequence of deterministic intermediate points $(k_n)_{n\in\NN}$ such that
\[
1\lqq k_n\lqq n\quad \mbox{satisfying } \quad k_n\nearrow \infty\ \quad \mbox{ and }\quad n^{\kappa-1}{k_n}\searrow 0 \quad \mbox{ as } n\to \infty \ .
\]
For instance, we can choose $k_n= n^{\gamma}$ for any $0<\gamma<1-\kappa$.
We then estimate
\begin{align}\label{eq: central sum}
n^\kappa\wnq^* =
n^\kappa\int_0^1 \Big(|F_n^{-1}(u)-F^{-1}(u)|^2 \wedge 1 \Big) du 
&\lqq n^{\kappa-1}{k_n} + n^\kappa\int_{\frac{k_n}{n}}^{1} \Big(|F_n^{-1}(u)-F^{-1}(u)|^2  \Big) du =: \mathcal E_{\kappa,n}\ .
\end{align}
The first term vanishes asymptotically by assumption.
The integral component of $\mathcal E_{\kappa,n}$ is then split into the 
integrals from $\frac{k_n}{n}$ to $\frac{1}{2}$ and from $\frac{1}{2}$ to $1$. 
The first integral is treated 
with the help of Theorem 2.4 in \cite{CH88}. 
With a slight adaption to our situation 
it states precisely that for the quantile process $Q_n(u)$ 
of the uniform distribution on $[0,1]$, given by Definition \ref{def: quantile process} 
as $Q_n(u)=\sqrt{n}(U_n^{-1}(u) -u)$, and any $\eta>1$ 
we have the following limit in distribution 
\begin{equation}\label{eq: CsH88b}
\Big(\frac{k_n}{n}\Big)^{2(\eta-1)}
\int_{\frac{k_n}{n}}^{\frac{1}{2}} \Big(\frac{Q_n(u)}{u^\eta}\Big)^2 du 
\stackrel{d}{\lra} \int_0^1 u^{2\eta- 4} |B_u|^2 du 
\quad { as }\quad n\ra \infty.
\end{equation}
The mean value theorem tells us that for some (random) intermediate value $\vartheta_n(u)$ satisfying
\begin{equation}\label{eq: intermediate}
u \wedge U_n(u)\lqq \vartheta_n(u)\lqq u\vee U_n(u)
\end{equation}
we have for each $\om \in \Om$ 
\begin{align*}
n \int_{\frac{k_n}{n}}^{\frac{1}{2}} \Big(|F_n^{-1}(u) - F^{-1}(u)|^2 \Big)  du 
& = n \int_{\frac{k_n}{n}}^{\frac{1}{2}} \bigg(\Big(\frac{\pd}{\pd u} F^{-1}(\vartheta_n(u))\Big)^2 |U_n^{-1}(u) - U^{-1}(u)|^2 \bigg)  du \\
& = \int_{\frac{k_n}{n}}^{\frac{1}{2}} \Big(\frac{|\sqrt{n} (U_n^{-1}(u)-u)|^2}{f(F^{-1}(\vartheta_n(u))^2} \Big) du\\
& = \int_{\frac{k_n}{n}}^{\frac{1}{2}} \Big(\frac{Q_n(u)^2}{f(F^{-1}(\vartheta_n(u))^2} \Big) du.
\end{align*}
By means of the facts $f(F^{-1}(u))= c^{-1}u^{1+\frac{1}{\alpha}}$ for $ u\in[0, 1]$ and $\eta = 1+ \frac{1}{\al}$, 
the bound \eqref{eq: intermediate} and the limit (\ref{eq: CsH88b}) yields 
\begin{equation}
\begin{aligned}
n\Big(\frac{k_n}{n}\Big)^{2/\alpha} \int_{\frac{k_n}{n}}^{\frac{1}{2}} \Big(|F_n^{-1}(u) - F^{-1}(u)|^2 \Big)  du 
&= \Big(\frac{k_n}{n}\Big)^{2/\alpha}c^2\int_{\frac{k_n}{n}}^{\frac{1}{2}} \Big(\frac{Q_n(u)^2}{\vartheta_n(u)^{2(1+\frac{1}{\alpha})}}  \Big)  du \\
&\lqq \Big(\frac{k_n}{n}\Big)^{2/\alpha}c^2\int_{\frac{k_n}{n}}^{\frac{1}{2}} \Big(\frac{Q_n(u)}{(u \wedge U_n)^{(1+\frac{1}{\alpha})}}  \Big)^2  du\\
&\stackrel{d}{\lra} c^2 \int_0^1 u^{2(\frac{1}{\alpha}-1)} |B_u|^2 du.
\end{aligned}
\end{equation}
Analogously to the above we obtain
\begin{align*}
n \int_{\frac{1}{2}}^{1} \Big(|F_n^{-1}(u) - F^{-1}(u)|^2 \Big)  du 
& = n \int_{\frac{1}{2}}^{1} \bigg(\Big(\frac{\pd}{\pd u} F^{-1}(\vartheta_n(u))\Big)^2 |U_n^{-1}(u) - U^{-1}(u)|^2 \bigg)  du \\
& \lqq \int_{\frac{1}{2}}^{1} \Big(\frac{|\sqrt{n} (U_n^{-1}(u)-u)|^2}{f(F^{-1}(\vartheta_n(u))^2} \Big) du.\\
& \lqq c^2 (\vartheta_n(\tfrac{1}{2}))^{-2(1+\frac{1}{\alpha})} \int_{\frac{1}{2}}^{1} \Big(\sqrt{n} (U_n^{-1}(u)-u)\Big)^2 du \\
&\stackrel{d}{\lra} c^2 (\tfrac{1}{2})^{-2(1+\frac{1}{\alpha})} \int_{\frac{1}{2}}^{1}  |B^0_u|^2 du < \infty \ ,
\end{align*}
where we used the monotonicity of $f(F^{-1}(u))$ and Theorem \ref{thm: lim_quantile_proc} (Glivenko-Cantelli).
\end{proof}

This theorem provides an upper bound for the coupling distance 
between the empirical law of $\mu$-distributed i.i.d. random variables with empirical 
measures $\mu_n$ and the distribution $\mu$ where they are drawn from.  
The mathematical formulation tells us that we can expect rates of convergence
for any $\kappa \in (0, \frac{\al}{\al+2})$ in our simulation studies 
compare the empirical rate of convergence to the fastest rate of convergence 
of order~$\frac{\al}{\al+2}$. 

\subsection{An overview of the estimation procedure}\label{sec: procedure}

Since the data of many real-world phenomena exhibit 
strong fluctuations in relatively short time, 
it is reasonable to extend stochastic models of continuous evolution 
to models with jumps. 
The easiest case to consider is a process of the type 
\begin{equation}\label{eq: billo model}
Y(t) = G(t) + L(t), \qquad t\in [0, T] \mbox{ for fixed }T>0,
\end{equation}
where $G = (G(t))_{t\in [0, T]}$ is a continuous process 
and $L = (L(t))_{t\in [0, T]}$ is a purely discontinuous L\'evy process. 
We have seen in Section \ref{sec: Levy and coupling} that $L$ 
is determined by a L\'evy triplet of the form $(0, 0, \nu)$, 
where $\nu$ is a L\'evy measure defined in (\ref{def: Levy mass}). 
For instance, the solutions of stochastic differential equations 
\[
Y(t) = x + \int_0^t f(Y(s)) ds + L(t), \qquad t\in [0, T],
\]
for globally Lipschitz continuous functions $f: \RR \ra \RR$ fall into this class. 
The aim of this procedure is  
to determine the nature of $L$ and 
thus of its L\'evy measure $\nu$. 

Given a data set $y = (y_{i})_{i=0, \dots n}$ from a time series 
the first step of our method consists in a non-dimensionalization of the data. 
A natural choice for the size of fluctuations is given by the interquartile range, 
through which we divide throughout this study. 
This procedure has the particular advantage that it allows us to compare different 
(non-dimensionalized) time series to the same class of L\'evy diffusions.  
We interpret the given time series $y$ as a realization of a 
(non-dimensionalized) process $Y$ given in the class of models (\ref{eq: billo model}) 
observed at discrete times $t_1< \dots < t_i < \dots < t_n$, that is 
\[
y_i = Y(t_i, \om) \qquad \mbox{ for some }\om \in \Om.  
\]
The following modeling assumptions determine the relation 
between the observed data and the underlying model.  
For any fixed threshold $\rho>0$: 
\begin{enumerate}
 \item The observation frequency is sufficiently high in comparison to the 
occurrence of large jumps given as increments beyond our threshold $\rho$ in 
that in each observed time interval $[t_{i-1}, t_i)$ at most 
one large jump occurs. That is, we do not see the sum of two or more large sums.  
 \item The behavior of small jumps is sufficiently benign in comparison to 
the large jump threshold $\rho$ during our observation. 
In particular, we assume that 
over one time interval $[t_{i-1}, t_i)$ 
small jump and continuous contributions cannot 
accumulate to this threshold and appear in the data as a large jump. 
\end{enumerate}

These assumptions can be made rigorous by further model assumptions on $Y$, 
such as the Lipschitz continuity of $G$.
Under the modeling assumptions 
it is justified to estimate jumps by increments 
\[
Y(t_i)- Y(t_{i-1}) \approx C^\rho(s)-C^\rho(s-0) \quad \mbox{ for exactly one }s\in [t_{i-1}, t_i),
\]
for the compound Poisson process $C^\rho$ given by equation (\ref{eq: Levy-Ito}), 
of the jumps of $L$ with $|\Delta_t L|>\rho$.
Hence the increments $Y(t_i)- Y(t_{i-1})$ can be considered 
as the realization of an i.i.d. sequence $X = (X_i)_{i=1, \dots, n}$ with some common law 
$\mu$ (concentrated on $\RR \setminus (-\rho, \rho)$). 
We denote by $x = (x_i)_{i=1, \dots n}$ the vector of large increments 
\[x_i = (y_i-y_{i-1}) \ind\{|y_i- y_{i-1}|> \rho\}\] 
so that 
\[
x_i = X_i(\om) \qquad \mbox{ for some }\om \in \Om.
\]
Let $\mu_n$ be the empirical measure of the data $X$ with respect to the common law $\mu$. 
The Glivenko-Cantelli theorem (\ref{eq:emp_dist_function}) tells us that for almost all $\om \in \Om$ 
\[
\mu_n(\om, \cdot) \ra \mu \qquad n \ra \infty, \mbox{ weakly. }
\]
Since by construction $\mu= \nu_\rho$ we obtain that $\mu_n(\om)$ converges to $\nu_\rho$ weakly. 
Since the Wasserstein distance encoded in the coupling distance 
metrizes the weak convergence we have for almost all $\om\in \Om$ 
\begin{equation}\label{eq: empLLN}
T_{\la_\rho}(\mu_n, \nu_\rho) \ra 0 \qquad \mbox{ as } n \ra \infty.  
\end{equation}
In particular we have an estimator for the tail of the L\'evy measure $\nu_\rho$. 
We are now in the position to estimate the distance 
between the L\'evy measure of the compound Poisson approximation $C^\rho$ 
of $L$ (respectively $Y$) with the true but unknown tail measure $\nu_\rho$ 
and the tail $\nu^*_{\rho^*}$ of 
a proposed L\'evy measure $\nu^*$. In particular we obtain 
\begin{equation}\label{eq: dreiecks}
T_{\la_{\rho^*}^*}(\nu^*, \nu_\rho) \lqq  T_{\la_\rho^*}(\nu^*_\rho, \mu_n) 
 + T_{\la_{\rho^*}^*}(\mu_n, \nu_\rho)
\lqq  (\la_{\rho^*}^* \wnq^*(\mu_n, \nu^*))^\frac{1}{2} + T_{\la_{\rho^*}^*}(\nu_\rho, \mu_n). 
\end{equation}
The last term on the right-hand side tends to $0$ by (\ref{eq: empLLN}) with the rate of convergence given in Theorem \ref{thm: rate of convergence} 
in Subsection \ref{subsec: 3.2}, 
while the first term can be calculated explicitly 
due to Section~\ref{subsubsec: empirical Wasserstein distance}. 

The first modeling assumption requires that the rate of jumps is relatively rare, 
we can achieve this by specifying $\la = 1$. 
This requirement along with the test model $\nu^*$ determines $\rho^*$.

We point out that 
the term on the left-hand side of equation (\ref{eq: dreiecks}) 
is an upper bound in Theorem \ref{Alex-theor-2} 
for the comparison between 
two models in the following sense.  
Under the interpretation that our data stem from 
a pure jump diffusion $X^i$ ($a_i = 0$ and $A_i = 0$) 
with (unknown) L\'evy measure $\nu$ we can compare 
its paths to the paths of the proposed model $\nu^*$ 
and quantify the corresponding coupling distances.

\section{Simulation studies} \label{sec: simulations}

\subsection{Rates of convergence for small observation lengths }\label{subsec: simulations rate of convergence}

In Subsection \ref{subsec: 3.2} we obtained upper bounds 
for the asymptotic rate of convergence of the statistic $\wnq^*(\mu)$ 
defined in (\ref{def: wns}) for large values of $n$. 
This section illustrates by means of simulations 
that the estimation procedure is unbiased and that 
the asymptotic rate of convergence 
for large $n$ holds approximately for small sample sizes driven by processes with power-law tails. 

We consider a parametrized compound Poisson 
process $C_t$ with unit intensity and polynomial jump distribution of order $\alpha_0$. 
More precisely we assume that each jump is distributed according to 
$\nu^{\alpha_0}$, where 
\begin{equation}
\nu ^\alpha(dz) = 
\begin{cases}
(\al \rho_0^\alpha)\frac{\ds dz}{\ds z^{1+\alpha}} , & z\gqq\rho_0\\
0 , & z<\rho_0
\end{cases}
.
\end{equation}
We fix $\rho_0 = 0.5$. 
Then for each given tail index parameter $\alpha_0\in\left\{1.4, 1.8, 3.0 \right\}$ 
we simulate a sample of $m = 100$ sample paths.
For each single sample path we evaluate 
the upper bound of the random function $\al \mapsto W_2(\nu^{\al_0}_n(\om), \nu^\al)$, 
given by the random curve 
\begin{equation}\label{eq: random curve statistics}
\alpha \mapsto \wnt(\al; \al_0)(\om) := \int_0^1 \big(|(F_n^{\al_0}(\om))^{-1}(u)-(F^{\al})^{-1}(u)|\wedge 1\big)^2 du,  
\end{equation}
for the first $n\in \{10^2,10^3,10^4,10^5\}$ points of the sample path. 
Here, $F_n^{\al_0}(\om)$ is the empirical distribution function of the first $n$ 
points of the sample $\om$ and $(F_n^{\al_0}(\om))^{-1}(u)$ is its corresponding quantile process, 
$F^{\al}$ is the exact empirical distribution function of $\nu^\al$ and $(F^{\al})^{-1}(u)$ is its 
corresponding quantile function. 
Note that for the parameter $\alpha = \alpha_0$ 
we have by construction $\wnt(\al_0; \al_0)(\om) = \wnq(\nu^{\al_0})(\om)$, 
as was derived in \eqref{eq:wnq} of Subsubsection \ref{subsubsec: empirical Wasserstein distance}, 
and for which we have obtained the rates of convergence in Subsection~\ref{subsec: 3.2}. 
This way we obtain $4$ sets (one for each value of $n$) of $m= 100$ 
realizations of the curves $\alpha\mapsto \wnt(\alpha; \alpha_0)(\om)$. 
These curves are shown in Figure~\ref{fig:chaosB} (a), (c), (e). 
Figures~\ref{fig:chaosB} (b), (d), (f) show the histograms of the minimum distance estimator $\widehat\alpha_n$,
\begin{equation}
\widehat\alpha_n(\om) := \arg\inf_{\alpha}\left\{\wnt(\al; \al_0)(\om)  \right\}\ ,
\label{eq:alphahat}
\end{equation}
indicating how the minima of the curves $\alpha \mapsto \wnt(\al; \al_0)(\om)$ are distributed around the true value $\al_0$. 

\hfill\\
\textit{$\lra$ Position of Figure 1} \\

\begin{figure}[ht]
\begin{center}
\subfloat[][]{\includegraphics[height=6cm,page=2]{./file_chaosB}}
\subfloat[][]{\includegraphics[height=6cm,page=1]{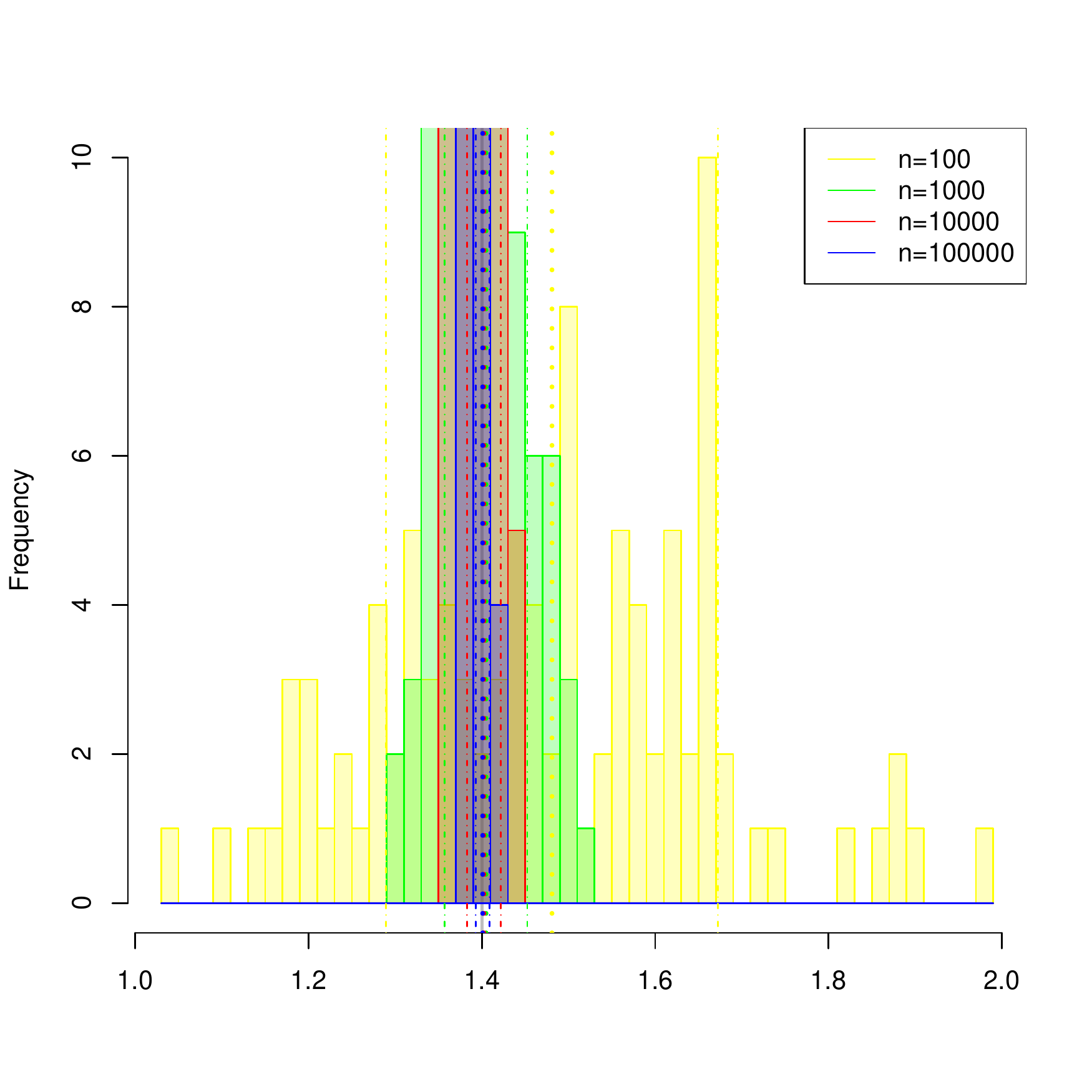}}
\\
\subfloat[][]{\includegraphics[height=6cm,page=4]{./file_chaosB}}
\subfloat[][]{\includegraphics[height=6cm,page=3]{./file_chaosB}}
\\
\subfloat[][]{\includegraphics[height=6cm,page=6]{./file_chaosB}}
\subfloat[][]{\includegraphics[height=6cm,page=5]{./file_chaosB}}
\end{center}
\caption{Simulations of the empirical distance curves $\alpha \mapsto \wnt(\al; \al_0)$ 
and the empirical minimal distance estimator $\widehat\alpha_n$ for distinct values of $\al_0$: 
(a) and (b) show $\al_0 =1.4$; (c) and (d) show $\al_0 = 1.8$; 
(e) and (f) show $\al_0 = 3.0$. }
\label{fig:chaosB}
\end{figure}

All sample curves in Figure \ref{fig:chaosB}(a),(c),(e) exhibit unique minima distributed around the true value $\alpha_0$. 
It is apparent that the minimum distances decrease with increasing observation length $n$. 
From Figure \ref{fig:chaosB} (b),(d),(f) it is clear that 
the empirical distribution of the minimum distance estimator $\widehat\alpha_n$ is centered at $\alpha_0$ with decreasing variance 
as $n$ increases. This behavior is also seen in Table \ref{tab:chaosB1}, which displays 
the empirical mean and empirical variance of the minimum distance estimator $\widehat\alpha_n$ 
and the minimum distance statistic $\widehat\wnq^*$ 
for the different observation lengths $n$ and $\alpha_0$. 

In order to compare the empirical and the asymptotic convergence rates of 
the empirical Wasserstein distance $\widehat\wnq^*$ to $0$ 
we compare the sample mean and variance of $\widehat \wnq^*$ 
to the asymptotic upper bound for the rate of convergence given in Theorem~\ref{thm: rate of convergence}. 
(Tables ~\ref{tab:quotient table for the sample mean of w} and \ref{tab:quotient table for the sample variance of w})
By Theorem~\ref{thm: rate of convergence} the statistic $\widehat\wnq^*$ converges for increasing 
$n$ with order $n^{-\kappa}$, for any $0< \kappa < \frac{\al}{\al+2}$. 
It is therefore instructive to compare 
the convergence of the sample mean of $\widehat\wnq^*$ to $0$ with $n^{-\frac{\al_0}{\al_0+2}}$ 
on a logarithmic scale. 
From inspection %
of the empirical means of $\widehat \wnq^*$ (the dotted horizontal lines in Figures \ref{fig:chaosB} (a),(c),(e)) 
it is apparent that the sample means 
for different powers of 10 are evenly spaced when scaled logarithmically. 
For increasing power of $10^i$, $i=2, \dots, 5$ we compare \[Q_{i,i+1}  = \mbox{sample mean}(\widehat{\mathrm{w}_{10^{i+1}}^*})
/ \mbox{sample mean}(\widehat{\mathrm{w}_{10^{i}}^*})\] with the respective power of $10^{-\frac{\alpha_0}{\alpha_0+2}}$ (Table \ref{tab:quotient table for the sample mean of w}). 
For the small values of $\al_0 =1.4$ and $\al_0 =1.8$ we obtain already for small sample sizes a reasonably good scaling. 
For $\al_0 =3.0$ 
the empirical ratio is of the same order of magnitude but larger than the asymptotic one. 
The linear scaling of the standard deviation of $\hat{w^*}$
makes it appropriate to compare the analogous quantity 
\[R_{i,i+1}  = \sqrt{\mbox{sample var}(\widehat{\mathrm{w}_{10^{i+1}}^*})
/ \mbox{sample var}(\widehat{\mathrm{w}_{10^{i}}^*})}\]
for the (sample) standard variation with the same quantity $10^{-\frac{\alpha_0}{\alpha_0+2}}$. 
We see a reasonable fit of the scaling factors for $\al_0 = 1.4$ and $\al_0 = 1.8$ and larger errors for $\al_0 = 3.0$. 
This overall picture indicates that the asymptotic rates of convergences are attained faster for smaller values of $\al_0$.

\hfill\\
\textit{$\lra$ Position of Table 1}\\ 

\begin{table}
\centering
\begin{displaymath}
\begin{array}{ll|ll|ll|ll|ll}
&  & & n=10^2 & & n=10^3 & & n=10^4 & & n=10^5\\
\alpha_0 &  &\widehat\alpha_n &\widehat{\wnq^*}&\widehat\alpha_n &\widehat{\wnq^*}&\widehat\alpha_n &\widehat{\wnq^*}&\widehat\alpha_n &\widehat{\wnq^*} \\ 
\hline
1.4 & \text{mean} & 
1.48 & 1.30 \times 10^{-1} & 1.40   & 4.74  \times 10^{-2} & 1.40  & 1.83 \times 10^{-2} & 1.40 & 6.63 \times 10^{-3}\\
& \text{var} & 
3.67 \times 10^{-3} & 2.71 \times 10^{-3} & 2.34 \times 10^{-3} & 2.02 \times 10^{-4} & 3.94 \times 10^{-4} & 2.86 \times10^{-5} & 6.48\times 10^{-5} & 3.99 \times 10^{-6} \\ 
\hline
1.8 & \text{mean} &
1.95 & 8.80 \times 10^{-2} & 1.81 \times 10^{-2} & 3.03 \times 10^{-2} & 1.81  & 1.05 \times 10^{-2} & 1.80 & 3.40 \times 10^{-3}  \\
& \text{var}  &
6.03 \times 10^{-2} & 1.54 \times 10^{-2} & 5.31 \times 10^{-3} & 0.12 \times 10^{-3}& 0.81 \times 10^{-3} & 1.37 \times 10^{-5} & 0.20 \times 10^{-3} & 1.29\times 10^{-6}\\ 
\hline
3.0 & \text{mean} &
3.17 & 5.08 \times 10^{-2} & 3.02 & 1.76\times 10^{-2} & 3.00 & 0.48 \times 10^{-2} & 3.00 & 0.16 \times 10^{-2}\\
& \text{var} &
0.20 & 0.70 \times 10^{-3} & 2.14 \times 10^{-2} & 9.14\times 10^{-5} & 0.32 \times 10^{-2} & 8.06\times 10^{-6} & 0.51 \times 10^{-3} & 7.72\times 10^{-7}
\end{array}
\end{displaymath}
\caption{Sample mean and variance of the empirical location of the minimum $\widehat \al_n$ and the 
empirical Wasserstein distance $\widehat\wnq^*$ 
for increasing observation length $n$ and different values of $\al_0$}
\label{tab:chaosB1}
\end{table}

\hfill\\
\textit{$\lra$ Position of Table 2} \\
\begin{table}
\begin{center}
\begin{displaymath}
\begin{array}{l|l|lll}
\alpha_0 & 10^{-\frac{\alpha_0}{\alpha_0+2}} & Q_{2,3} & Q_{3,4} & Q_{4,5} \\ \hline
1.4 & 0.39 & 0.37 & 0.39 & 0.32\\ \hline
1.8 & 0.34 & 0.34 & 0.35 & 0.32\\ \hline
3.0 & 0.25 & 0.35 & 0.27 & 0.33
\end{array}
\end{displaymath}
\caption{Quotient $Q_{i,i+1}  = \mbox{sample mean}(\widehat{\mathrm{w}_{10^{i+1}}^*})
/ \mbox{sample mean}(\widehat{\mathrm{w}_{10^{i}}^*})$}
\label{tab:quotient table for the sample mean of w}
\end{center}
\end{table} 

\hfill\\
\textit{$\lra$ Position of Table 3} \\
\begin{table}
\begin{center}
\begin{displaymath}
\begin{array}{l|l|lll}
\alpha_0 & 10^{-\frac{\alpha_0}{\alpha_0+2}} & R_{2,3} & R_{3,4} & R_{4,5} \\ \hline
1.4 & 0.39 & 0.31 & 0.34 & 0.37 \\ \hline
1.8 & 0.34 & 0.28 & 0.34 & 0.31\\ \hline
3.0 & 0.25  & 0.37 & 0.30 & 0.31
\end{array}
\end{displaymath}
\caption{Quotient $R_{i,i+1}  = \sqrt{\mbox{sample var}(\widehat{\mathrm{w}_{10^{i+1}}^*})
/ \mbox{sample var}(\widehat{\mathrm{w}_{10^{i}}^*})}$  
}
\label{tab:quotient table for the sample variance of w}
\end{center}
\end{table}

\subsection{Sensitivity analysis for different cutoffs } \label{subsec: cutoff sensitivity}

In the previous subsection we fixed the cutoff $\rho = 0.5$ 
and varied the values of $\al_0$ and $n$ 
in order to study the convergence of the minimal distance of the curves 
$\alpha \mapsto \wnt(\al; \al_0)(\om)$.  
We will now present a sensitivity analysis of these curves for different values of $\al_0$ and higher cutoffs $\rho$, 
for fixed observation number $n=10^5$. 
In the case of a pure power tail L\'evy measure, 
increasing the cutoff has the effect of choosing a random subsample of the original data, 
with the same tail. 
Therefore, we expect increasing $\rho$ to yield samples of curves with 
minima at the same location but taking larger values. 

Figure \ref{fig:chaosC2} illustrates the 
sensitivity of the distance functional 
to different cutoff values. 
We fix $\rho_0 = 0.5$. 
Then for each value $\al_0 = \{0.7, 1.4, 1.8, 2.3, 3.0, 4.0\}$ we simulate a sample of 
$m= 100$ noise realizations $\om$ with $n= 10^5$ jump increments and 
calculate the functions $\al \mapsto \wnt(\al; \al_0)(\om)$ 
for those increments $x_i$  such that $|x_i| \gqq \rho$, 
separately for the different values $\rho = \{0.5, 0.7, 1.4, 1.8\}$. 
For each of the values $\al_0$ we obtain 
$4$ sets of $m=100$ such curves, 
illustrated using different colours for each value of $\rho$. 

In Figure \ref{fig:chaosC2} it is evident that for fixed $\al_0$ 
the position of the minima of the curves does not depend on the cutoff values. 
However, while for fixed cutoffs the shape of the curves remains essentially the same 
(cf. Figure \ref{fig:chaosB} (a),(c),(e)), 
we see a change of shape for different cutoffs, best seen in picture Figure \ref{fig:chaosC2}(f). 
The logarithmic vertical scale hence 
indicates a polynomial growth behavior. 
Larger cutoff values generally result in both larger values of the coupling distance 
and wider troughs around the minimal value of $\alpha$, 
while larger values of $\al_0$ produce flatter curves 
with broader sampling distributions of the minima.

\subsection{Estimator behavior for non-polynomial Gaussian data} 

For the applications to empirical 
data sets it is important to show that the curves 
$\alpha \mapsto \wnt(\al; \al_0)(\om)$
distinguish clearly between data coming from distributions 
with and without a polynomial tail. 
As prototype non-polynomial data we use standard Gaussian data. 

Figure \ref{fig:CPPnoise} illustrates on the left side 
the increments of one single trajectory 
of a compound Poisson process with polynomial L\'evy measure $\nu^{\al_0}$ 
for tail index $\al_0= 1.6$, intensity $1$ and $\rho_0 = 0.5$ 
with $n=10^4$ increments, while on the right side we see the increments 
of a compound Poisson process with standard Gaussian L\'evy measure $N(0,1)$, 
intensity $1$ and $\rho_0 = 0.1$. 
In both cases, the horizontal axis is not continuous time but the discrete set 
of Gamma arrival times at which the jumps occur. 

Figure \ref{fig:Simulation-cutoff-analysis} 
shows on the left side 
the family of curves $[1.0, 2.0] \ni \alpha \mapsto \wnt(\al; \al_0, \rho)(\om)$, 
obtained from the compound Poisson process with polynomial L\'evy measure 
as the cutoff $\rho$ varies from $0.1$ to $0.7$. 
The thick line is the locus of minima of the individual curves. 
The plot on the right side of this Figure shows the family of curves 
\begin{equation}\label{eq:Gaussian_functional}
[2.0, 14.0] \ni \alpha \mapsto \int_0^1 \big(|(F_{10^4}^{N(0,1)}(\om))^{-1}(u)-(F^{\al})^{-1}(u)|\wedge 1\big)^2 du (\al; \al_0, \rho)(\om),
\end{equation}
where $u\mapsto (F_n^{N(0,1)}(\om))^{-1}(u)$ is the empirical quantile function of the increments of the compound Poisson process with the Gaussian L\'evy measure, 
estimated with $n$ realizations of $N(0,1)$ increments at realization $\om$. 
This functional is the analogue of $\wnt(\al; \al_0, \rho)$ defined in (\ref{eq: random curve statistics}) for the standard normal distribution 
$N(0,1)$. 

It is striking 
that for the polynomial L\'evy process, the locus of the curve minima of $[1.0, 2.0] \ni \alpha \mapsto \wnt(\al; \al_0, \rho)(\om)$ 
moves towards larger values of $\alpha$ with increasing $\rho$ 
until the original value $\al_0 = 1.6$ is reached at $\rho_0 = 0.5$. 
In addition the values of the minimum coupling distance 
decrease across this parameter range before dropping sharply just before $\rho = \rho_0 = 0.5$. 
For values of $\rho$ above $\rho_0 = 0.5$, 
the position of the minimum does not change while its value increases slightly. 
These results imply that the locus of minimal distances for a perfectly simulated power law  
attains its global minimum at the true cutoff value $\rho_0$ and the true power law index $\alpha_0$ of the simulation. 
Figure \ref{fig:Simulation-cutoff-analysis} also presents the number of observations 
used in the curve estimate as a function of the cutoff. 
For cutoffs which are smaller than $\rho_0$ all points are used by the estimator.  
As $\rho$ increases above $\rho_0$, the number of points entering the calculation decreases. 

The corresponding figure for the Gaussian L\'evy process 
shows that in the case of a non-polynomial behavior the locus of minima 
for different cutoffs $\rho$ 
shows no distinct minimum. 
While a weak local minimum is reached at $\alpha \approx 8$, 
rather than staying near this value as $\rho$ increases 
the locus of minima continues to move towards the right. 
For sufficiently large values of the cutoff, the minimum coupling distances 
for individual curves start to behave erratically due to small sample sizes. 

Figure \ref{fig:scaling} presents a plot 
of the family of curves $[1.0, 2.0] \ni \alpha \mapsto \wnt(\al; \al_0, \rho)(\om)$, 
as in the left Figure \ref{fig:Simulation-cutoff-analysis}, 
with the following difference: we have rescaled the data by the factor $2$, 
the lower cutoff has been taken as $\rho = 2\times 0.5$, and the constant $s$ in 
the coupling distance is taken to be $s=2^2$. 
It can be checked easily that the cutoff distances in this family of curves scales as $2^2 = 4$. 
We see that by appropriate rescaling of the data, the cutoff $\rho$ and the coupling 
constant $s$ we obtain exactly the same curve as the one on the left side of Figure \ref{fig:Simulation-cutoff-analysis}, 
only with appropriately rescaled values of the coupling distance. 
As a consequence our procedure allows to adopt a 
systematic non-dimensionalization of all data under consideration and 
to fix the coupling constant at $s=1$.  
Specifically we will rescale all data under consideration divided by their interquartile range. 
This measure is a particularly convenient measure of width of the distribution 
since it does not require the existence of moments. 
A second consequence consists in the fact that the value of the minimum distance 
depends on the scale of the data (e.g. on the units) 
and only the shape of the curves gives meaningful information. 

We have introduced the coupling distance as a truncated 
version of the Wasserstein distance, 
in order to accommodate processes that 
do not possess second-order moments, such as $\alpha$-stable processes. 
For Figure~\ref{fig:WasserVsCoupling} we have simulated 
a compound Poisson process with $\rho = 0.5$ and $\alpha_0 = 3.6$ 
in order to account for the existence of second moments and hence the Wasserstein-$2$ distance. 
We show plot of the family of curves $[2.0, 5.0] \ni \alpha \mapsto \wnt(\al; 3, \rho)(\om)$, 
where $\rho$ varies from $0.2$ to $0.8$, 
for values of $s=1$ (our standard setting) and $s= \infty$ (corresponding to the non-truncated Wasserstein distance). 
We recognize that both measures pick as minimal distance location 
a value close to $\al_0 = 3.5$. However, it is clearly visible that the Wasserstein-$2$ distance 
(which does not exist for $\al = 2$) has a pole at this value, which implies that 
for values $\al_0$ close to $2$ the rates of convergence which are derived in Section \ref{sec: simulations} 
cannot hold. As seen in Figure \ref{fig:Simulation-cutoff-analysis} the coupling distance 
works properly also for values $\al_0 \lqq 2$ and is hence more general. 
This fact is particularly useful in order to detect true $\alpha$-stable power laws, where $\alpha < 2$. 

\hfill\\
\textit{$\lra$ Position of Figure 2}\\


\begin{figure}[ht]
\begin{center}
\includegraphics[height=6cm]{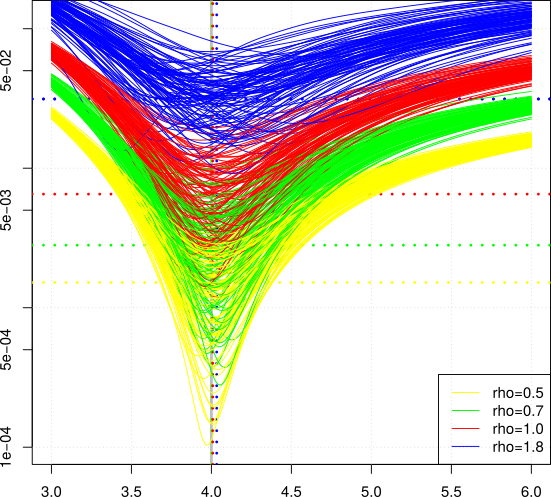}
\includegraphics[height=6cm]{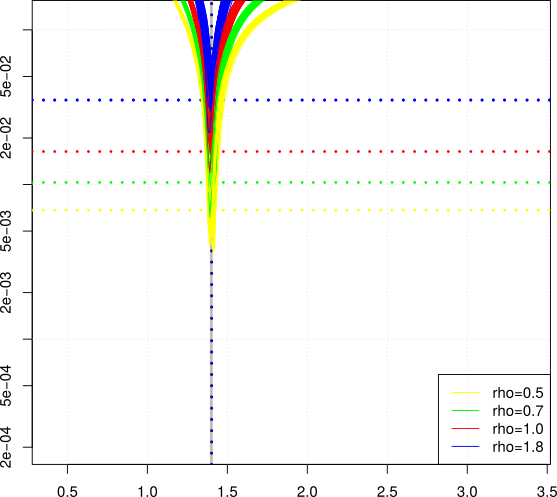}
\\
\includegraphics[height=6cm]{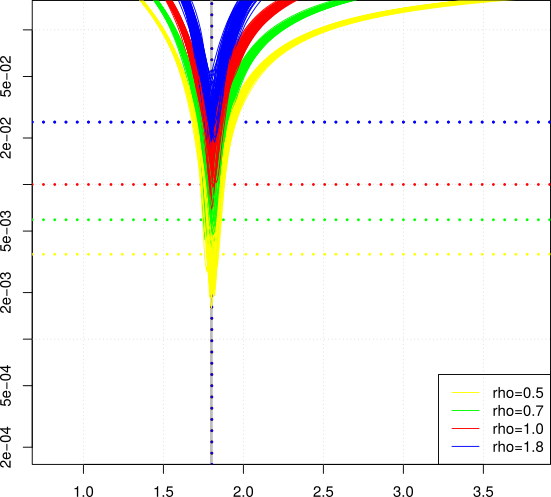}
\includegraphics[height=6cm]{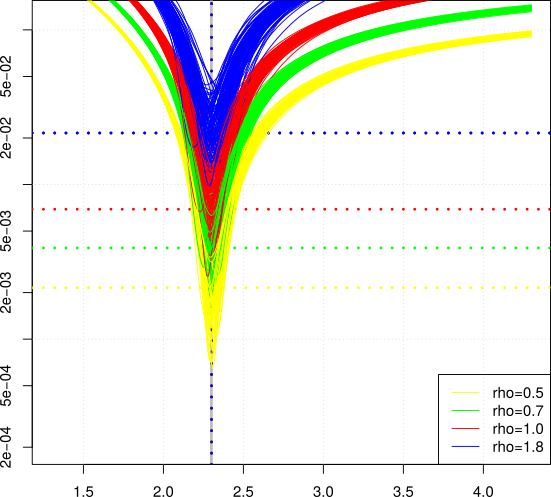}
\\
\includegraphics[height=6cm]{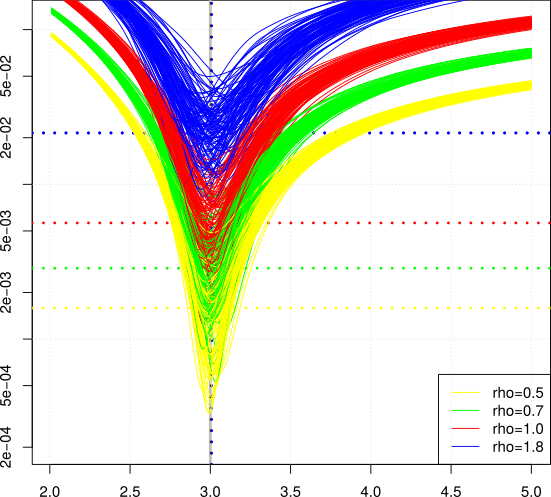}
\includegraphics[height=6cm]{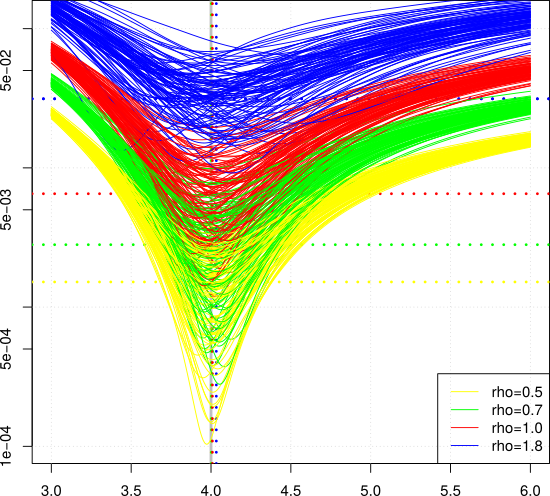}
\end{center}
\caption{Plots (a) - (f) correspond to $\al_0 \in \{0.7, 1.4, 1.8, 2.3, 3.0, 4.0\}$, 
where on each plot you see four samples of $m=10$ curves 
with the cutoffs $\rho \in \{0.5, 0.7, 1.0, 1.8\}$ which are color coded}
\label{fig:chaosC2}
\end{figure}

\hfill\\
\textit{$\lra$ Position of Figure 3}\\

\begin{figure}[ht]
\begin{center}
\includegraphics[height=8cm]{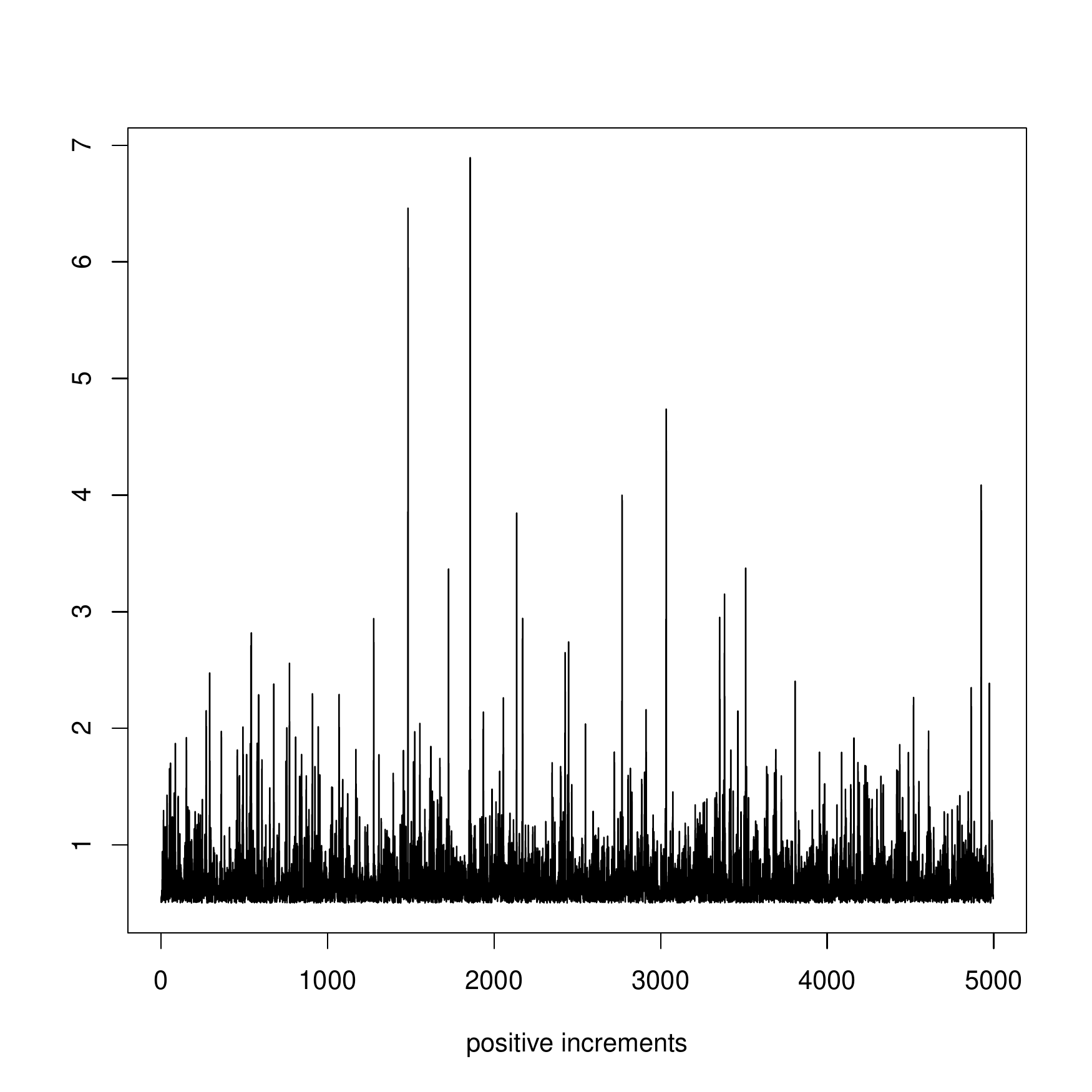}
\includegraphics[height=8cm]{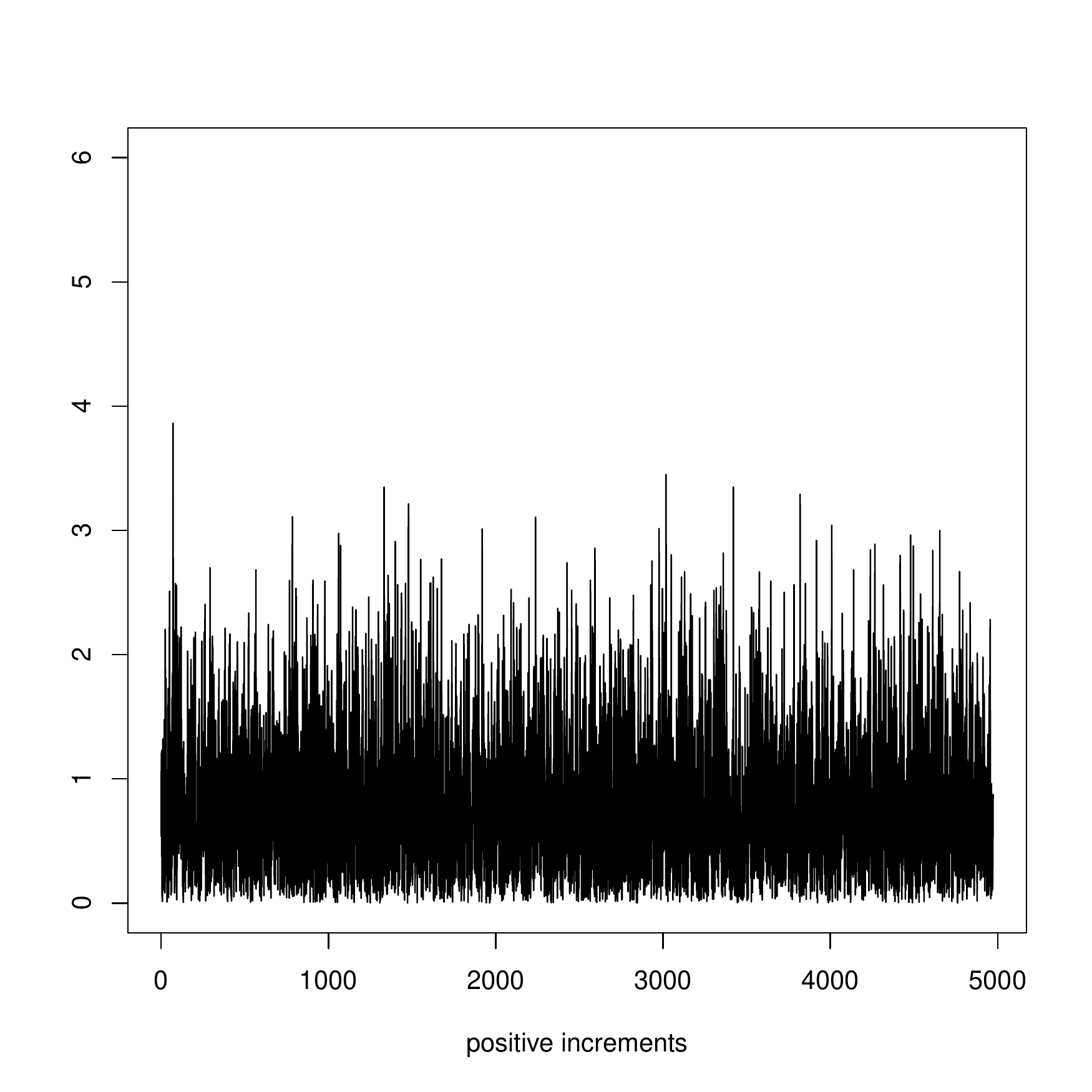}
\end{center}
\caption{Simulation of the increments of a single realization of a 
compound Poisson process of intensity $1$ for 
tail index $\al_0 = 1.6$ and lower cutoff $\rho  = 0.5$ (left side) 
a compound Poisson process of intensity $1$ with standard normal $N(0,1)$ jump measure (right side). 
}
\label{fig:CPPnoise}
\end{figure}

\hfill\\
\textit{$\lra$ Position of Figure 4}\\ 

\begin{figure}[ht]
\begin{center}
\includegraphics[height=8cm,page=1]{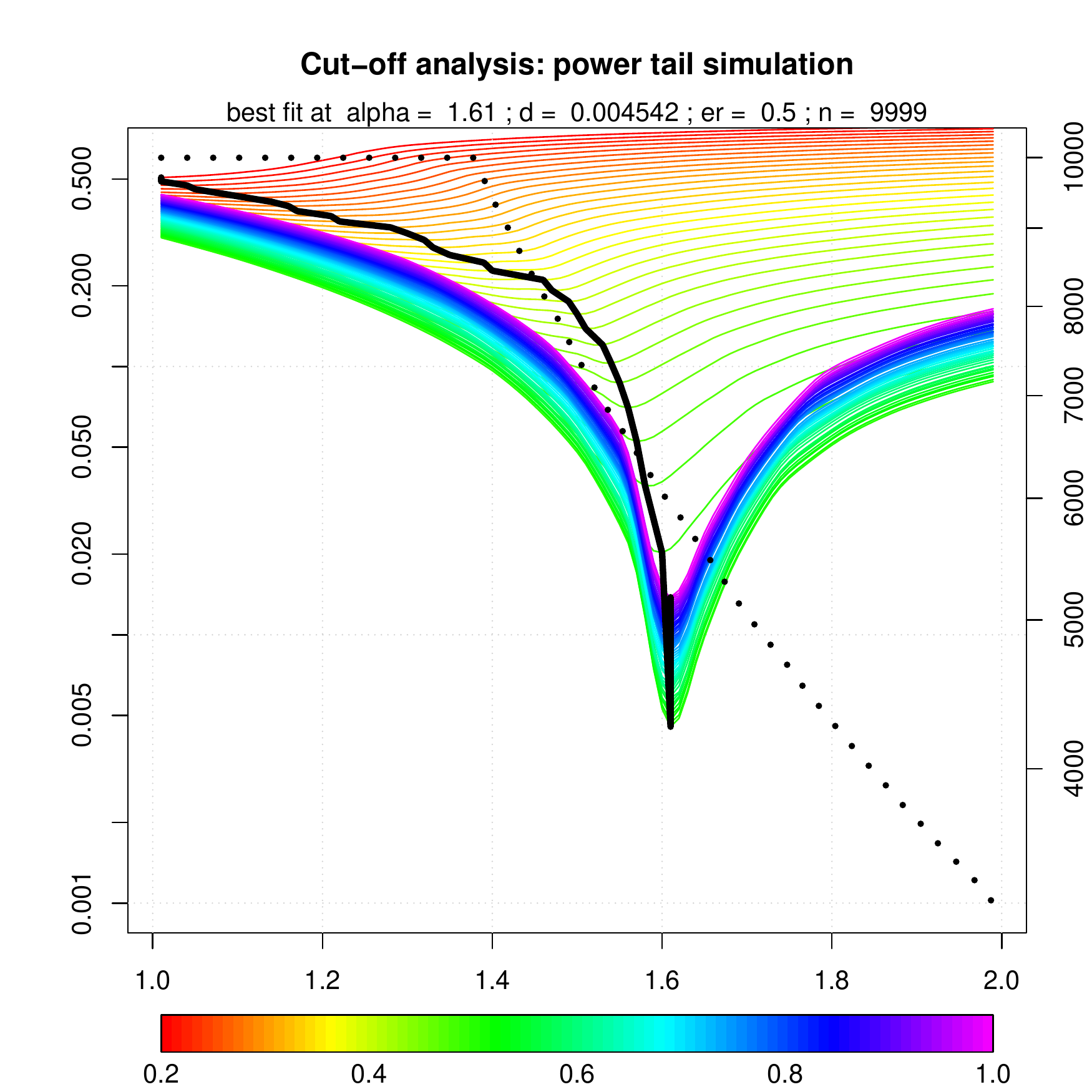}
\includegraphics[height=8cm,page=1]{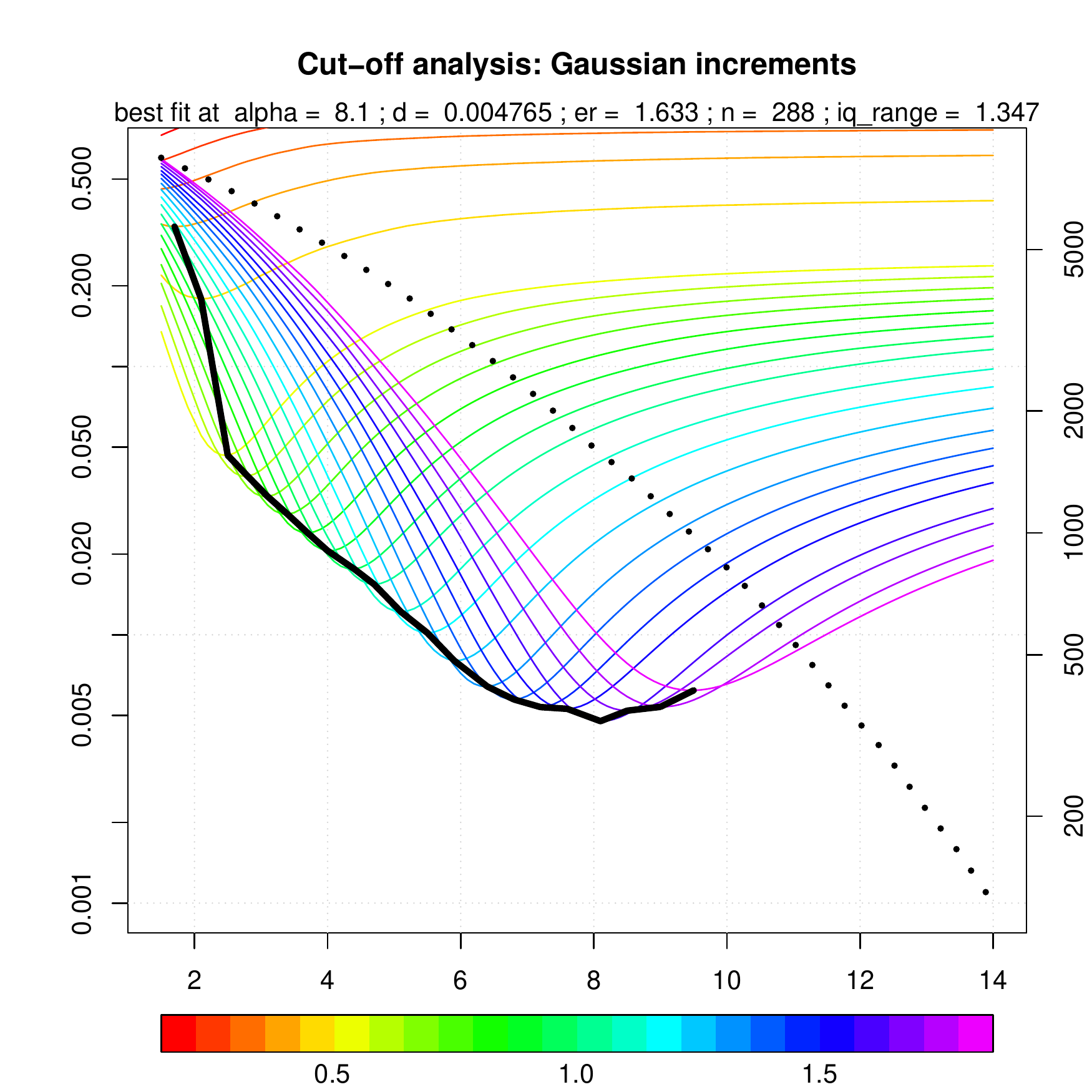}\\
\end{center}
\caption{Cutoff sensitivity analysis for simulated data of a 
compound Poisson process of intensity $1$ with renormalized jump measure $\nu^{\al_0}$ 
with tail index $\al_0 = 1.6$ and lower cutoff $\rho  = 0.5$ (left side) and 
a compound Poisson process of intensity $1$ with standard normal $N(0,1)$ jump measure (right side). 
The number of points included for each colored curve 
is color coded as follows. 
If one picks a curve and detects its color (dark blue, say) 
and looks for the value of the black dashed line at the position of that (dark blue) color 
on the color spectrum one can read off 
on the right side the number of included points. 
}
\label{fig:Simulation-cutoff-analysis}
\end{figure}

\hfill\\
\textit{$\lra$ Position of Figure 5}\\

\begin{figure}[ht]
\begin{center}
\includegraphics[height=8cm,page=2]{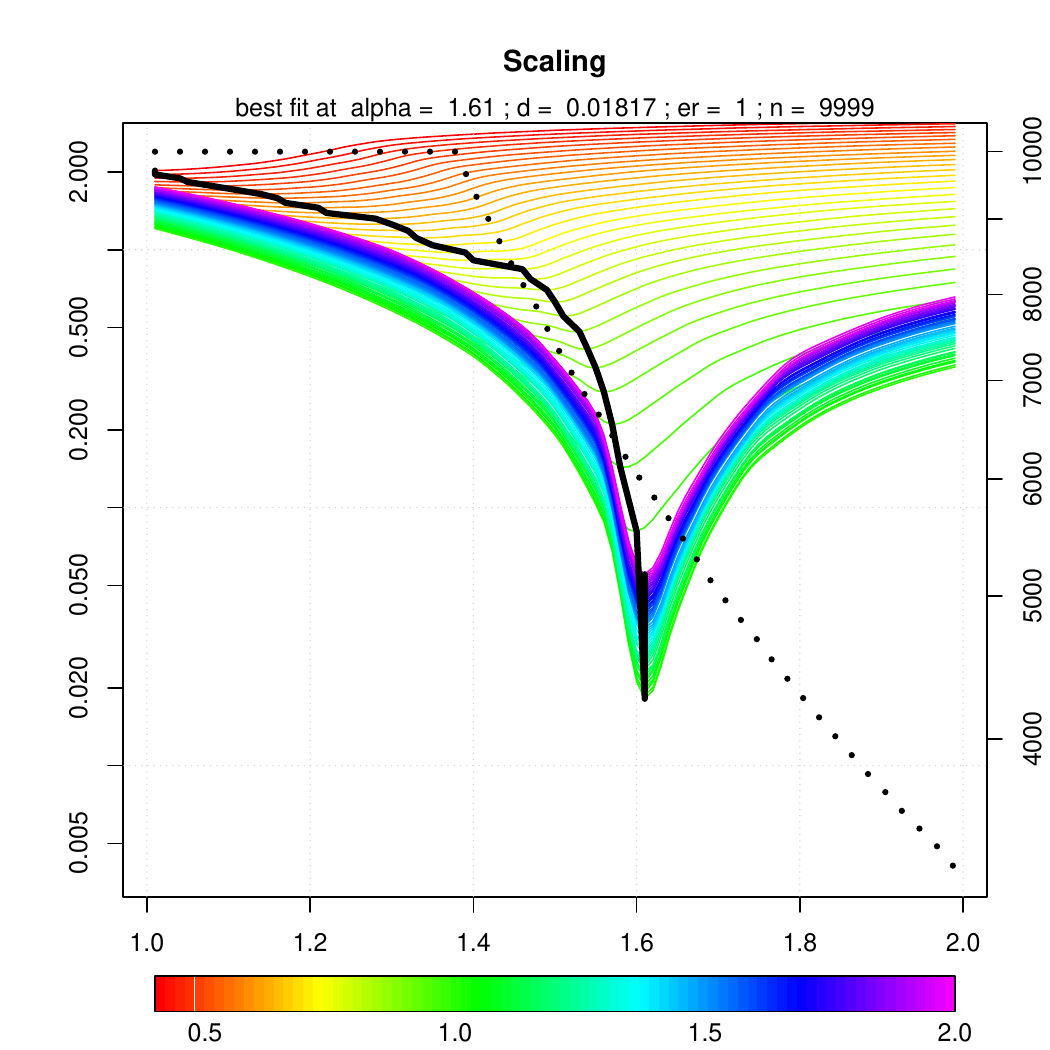}
\end{center}
\caption{Scaling property of the coupling distance: 
Cutoff sensitivity analysis for simulated data of a 
compound Poisson process of intensity $1$ for with renormalized jump measure $\nu^{\al_0}$ 
with tail index $\al_0 = 1.6$. 
Here the data were multiplied with $2$, the lower cutoff is taken $\rho  = 2\times 0.5$ and the coupling distance cutoff $s=2^2$. 
}
\label{fig:scaling}
\end{figure}

\hfill\\
\textit{$\lra$ Position of Figure 6}\\ 

\begin{figure}[ht]
\begin{center}
\includegraphics[height=8cm,page=1]{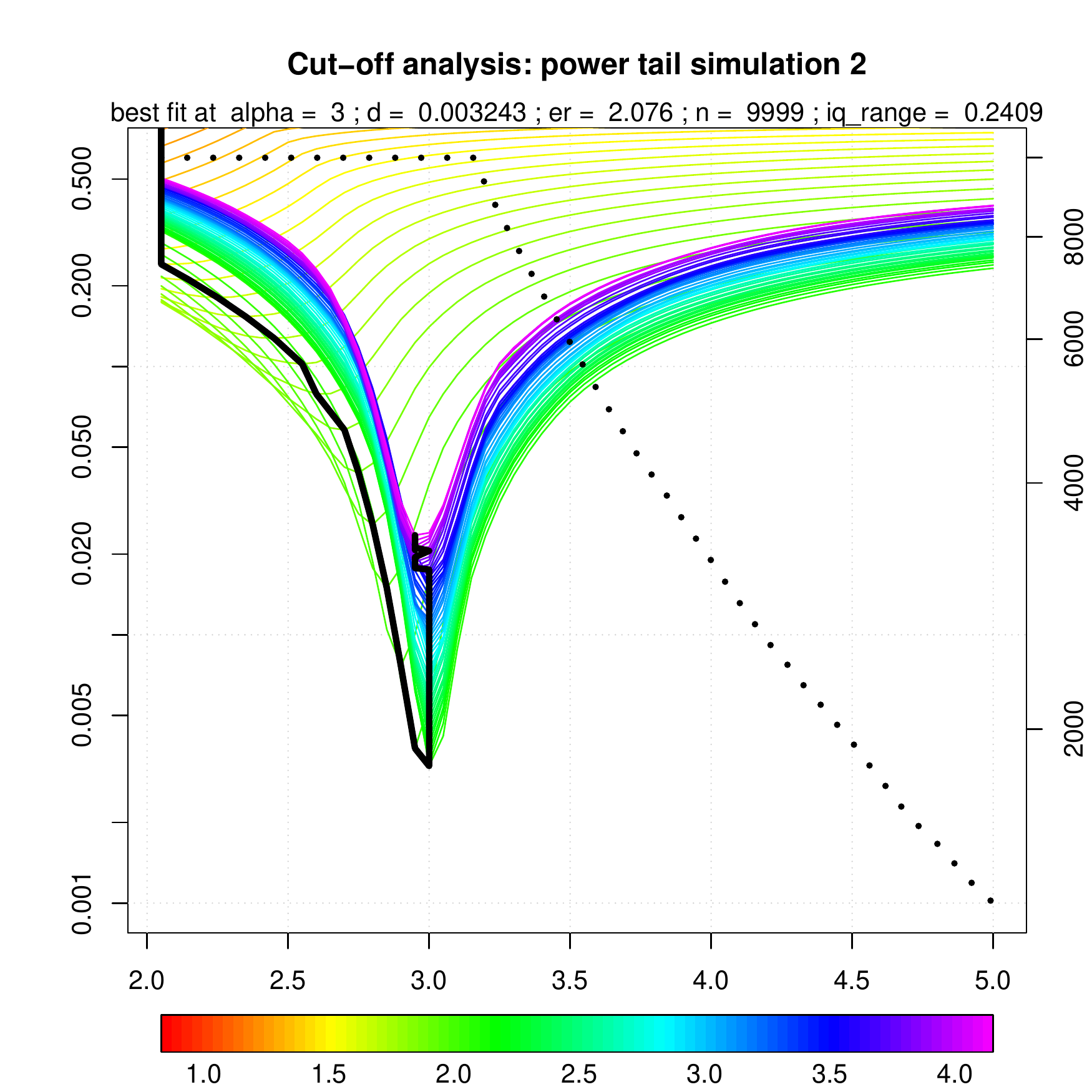}
\includegraphics[height=8cm,page=2]{./WassersteinVsCoupling}\\
\end{center}
\caption{Cutoff sensitivity analysis for simulated data of a 
compound Poisson process of intensity $1$ with renormalized jump measure $\nu^{\al_0}$ 
with tail index $\al_0 = 1.6$ and lower cutoff $\rho  = 0.5$ (left side) and 
a compound Poisson process of intensity $1$ with standard normal $N(0,1)$ jump measure (right side)}
\label{fig:WasserVsCoupling}
\end{figure}

\section{Application to example climatic data sets} \label{sec: climate examples}

In this section we will apply the estimators discussed above to two sets of climatic data. 
Because the estimators as constructed assume stationary stochastic increments, 
that is essentially additive noise models with autonomous coefficients, 
the datasets we consider will be those for which the assumption of this increment stationarity is reasonable. 

The first example which we discuss 
is a paleoclimatic time series of calcium concentration proxies (cf. \cite{FNAM93})
which has received considerable interest in recent years.  
The second class of examples are data sets of precipitable water vapor 
from different measurement stations in the Western Tropical Pacific.  

\subsection{Paleoclimatic GRIP data }\label{subsec: GRIP}

Paleoclimate proxy data from Greenland ice cores show clear evidence of two distinct states of high-latitude 
climate during the past glacial period (e.g. 
\cite{MAW08}). Residence time within both 
the relatively warm interstadial state and the relatively cool stadial state are in on order of $1000$ years 
(e.g. 
\cite{Dit99a},
\cite{Ra03}). Using the logarithm of the calcium concentration record (a measure of continental 
aridity which correlates well with proxy estimates of high-latitude temperature) from the 
Greenland Ice Core Project (GRIP), Ditlevsen \cite{Dit99a} argued that the statistics of transitions between 
states shows evidence of $\alpha$-stable driving noise with $\alpha = 1.75$. 

\hfill\\
\textit{$\lra$ Position of Figure 7}\\

\begin{figure}[ht]
\begin{center}
\includegraphics[height=10cm]{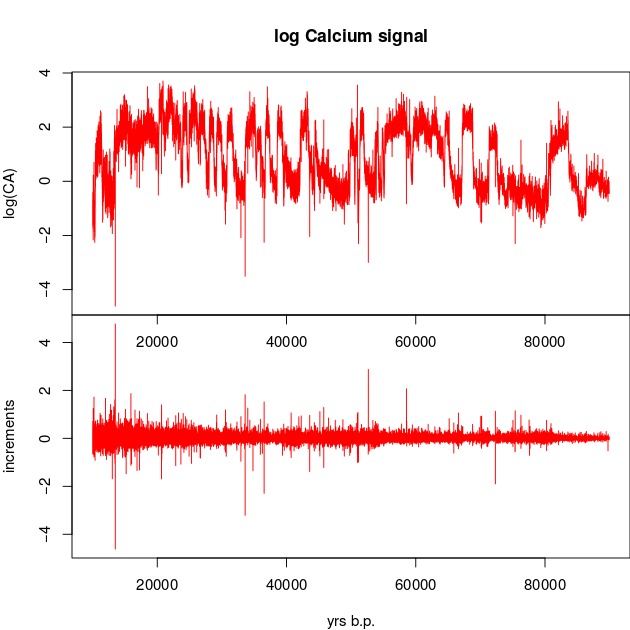}
\caption{Paleoclimatic GRIP data}
\label{fig:logCA}
\end{center}
\end{figure}

Motivated by the analysis in \cite{Dit99a} 
there has been a series of works 
on the model selection problem (\cite{GI14}, \cite{HIP08}, \cite{HIP09}), 
using different methods to estimate the tail index, 
which led to indications of $\al = 0.75$ (that is following Section 2 an $\alpha$-stable process even without first moments) 
and $\alpha = 1.75$ according with Ditlevsen's proposal. 
These methods have been based on the self-similar scaling of $\alpha$-stable processes, 
which turns out to be a rather fragile property, 
since it enhances small errors to large scales. 

In \cite{GHKK16} we applied an earlier version of the estimation procedure described in this study to the same 
GRIP log calcium time series analyzed in \cite{Dit99a}. 
Local best fits were obtained by tuning the parameters $\la$, $\rho$ and $\al$ \
for the positive tail at the cutoff $\rho^+ = 0.36$ and the index $\al_0^+ = 3.6$ 
and for the negative tail at $\rho^- = 0.35$ and the index $\al_0^- = 3.55$ with 
reasonable sample sizes $n^+ = 894$ and $n^- = 530$ and small distances $d^+ = 0.081 \ll 1$ and $d^- = 0.089 \ll 1$. 
A systematic optimization of the cutoff and the tails was not considered in this earlier study. 

In this subsection we will re-analyze this time series using 
the procedure described in the second part of Subsection~\ref{subsec: cutoff sensitivity}. 
First we standardize the tail intensity to $\la= 1$. 
We then search for the pair $(\rho, \alpha_0)$ such that the minimum of the curve 
$\alpha \mapsto \wnt(\al; \al_0, \rho)(\om)$ is minimal over all those parameter values. 

\hfill\\
\textit{$\lra$ Position of Figure 8}\\ 

\begin{figure}[ht]
\begin{center}
\subfloat[][]{\includegraphics[height=8cm,page=1]{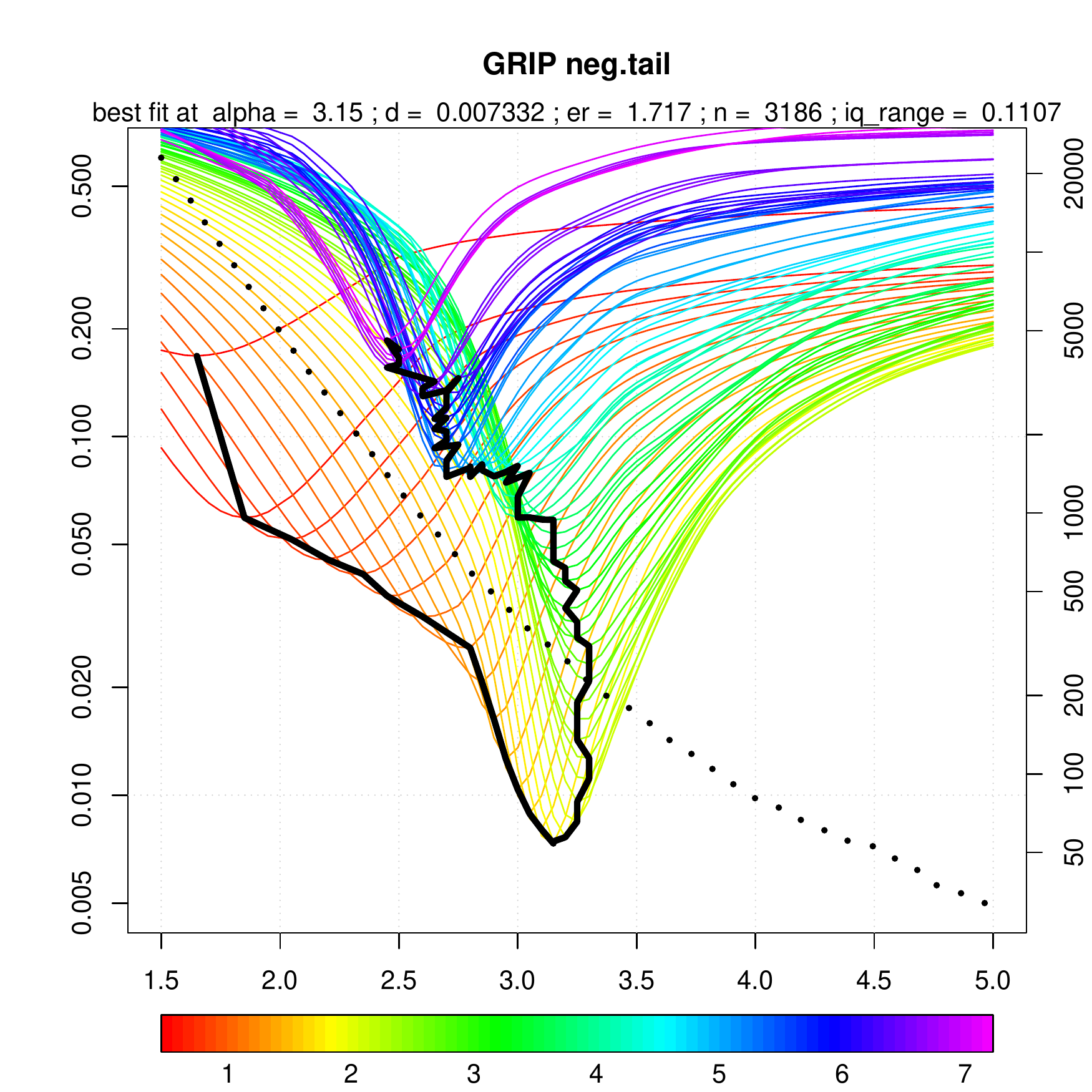}}
\subfloat[][]{\includegraphics[height=8cm,page=2]{./Grip_epsilon_sens}}
\end{center}
\caption{Plots of the functions 
$\alpha \mapsto \wnt(\al; \al_0, \rho)(\om)$ the negative and positive measure $\nu^{\al_0}$ for different parameters $\rho, \alpha_0$}
\label{fig:GRIP_epsilon_sens}
\end{figure}
 
The GRIP log calcium time series and that of its increments are shown in Figure \ref{fig:logCA}.  
Both panels of Figure \ref{fig:GRIP_epsilon_sens} show the functions $\alpha \mapsto \wnt(\al; \al_0, \rho)(\om)$ 
for different cutoffs $\rho$ ranging 
from $0.5$ (orange-red) to $7$ (pink). 
The left side of Figure \ref{fig:GRIP_epsilon_sens} 
treats the negative tails of the empirical increment distributions 
and the right side the positive tails. 
For the negative tails the locus of minima yields a very distinctive 
minimum for a cutoff value of $\rho^- = 1.71$ 
with approximately $n=3186$ data points 
and with a minimal distance at $d^- = 0.00732$  
and the best fit for the tail index $\al^- \approx 3.15$. 
The minimal distance is of the magnitude 
of the simulated compound Poisson data of Figure \ref{fig:Simulation-cutoff-analysis}(left side) 
and indicates a clear polynomial behavior. 
For the positive tails this picture is even more pronounced with 
an even sharper minimum for approximately the same cutoff value $\rho^+ \approx 3.2$ 
with approximately $n= 1260$ data points 
and minimal distance of $d^+ = 0.00987$ 
and estimated tail index $\alpha^+ = 3.2$. 
For larger values of the cutoff, the curve of minima moves upward 
rather than laterally until the number of data points entering the analysis becomes 
very small. This relative insensitivity of the estimated tail index 
to the cutoff values is further evidence of polynomial tails in the driving noise. 

The results of this analysis provide strong evidence for polynomial tails 
in the driving process in these data, consistent with the results of \cite{HIP08}, \cite{GI14}, \cite{Ga11}.

However, consistent with our earlier estimates in \cite{GHKK16}, 
we find that the tail parameter $\alpha$ is greater than two, 
so the increment distribution of the driving noise process has finite variance. 
The detection of tail behaviour with $\al> 2$ in these data is noteworthy.
Classical homogenization or stochastic averaging results such as \cite{PK74} 
indicate that for coupled dynamics in which a slow variable is driven by 
a finite variance fast variable, 
the asymptotic dynamics in the limit of infinite scale separation should be a standard stochastic differential equation 
with Brownian noise. The presence of driving noise with finite-variance, polynomial-tailed increments in the GRIP data is suggestive 
that the timescale separation between fast and slow processes is not sufficiently large in this setting for the asymptotic 
dynamics to hold.

\subsection{Western Tropical Pacific precipitable water vapor data }

The second observational data we consider are 
two time series of hourly-mean precipitable water vapor (PWV, the total water content of the atmospheric column, in cm) 
from the islands Nauru ($0.5^{\circ}$ S, $166.9^\circ$ E) and Manus ($2.1^{\circ}$ S, $147.4^\circ$ E) in the Western 
Tropical Pacific. The Nauru PWV time series extends from January 1 1998 to December 31 2010, 
and the Manus time series from January 1 1996 to December 31 2011. These data were obtained from the ARM 
(Atmospheric Radiation Measurement) Best Estimate Data set (\cite{Xi10}), and downloaded from www.archive.arm.gov.

These data were investigated because the onset of convective precipitation is a threshold process that occurs 
when at least part of the atmospheric column becomes saturated with regards to water vapor. 
While in the non-precipitating state PWV can be considered to undergo gradual random variations, the transition to the 
precipitating state produces a rapid decrease in column water content. 
Recent analysis \cite{PDCN10} has shown 
that the probability distribution of rainfall event 
magnitudes has a power-law tail, as do dry spell duration and (to a lesser extent) precipitation event duration. 
Furthermore, \cite{SN11} and \cite{SN14} 
have demonstrated that a simple two-state precipitation on and off stochastic differential equation model for column moisture can reproduce the 
observed statistics of precipitation. The jump-like changes in PWV resulting from the rapid onset of precipitation suggest that on sufficiently coarse 
timescales the driving process of PWV could be represented as an alpha-stable process (for at least the negative increments).

\hfill\\
\textit{$\lra$ Position of Figure 9}\\

\begin{figure}[ht]
\begin{center}
\subfloat[][]{\includegraphics[height=8cm,page=1]{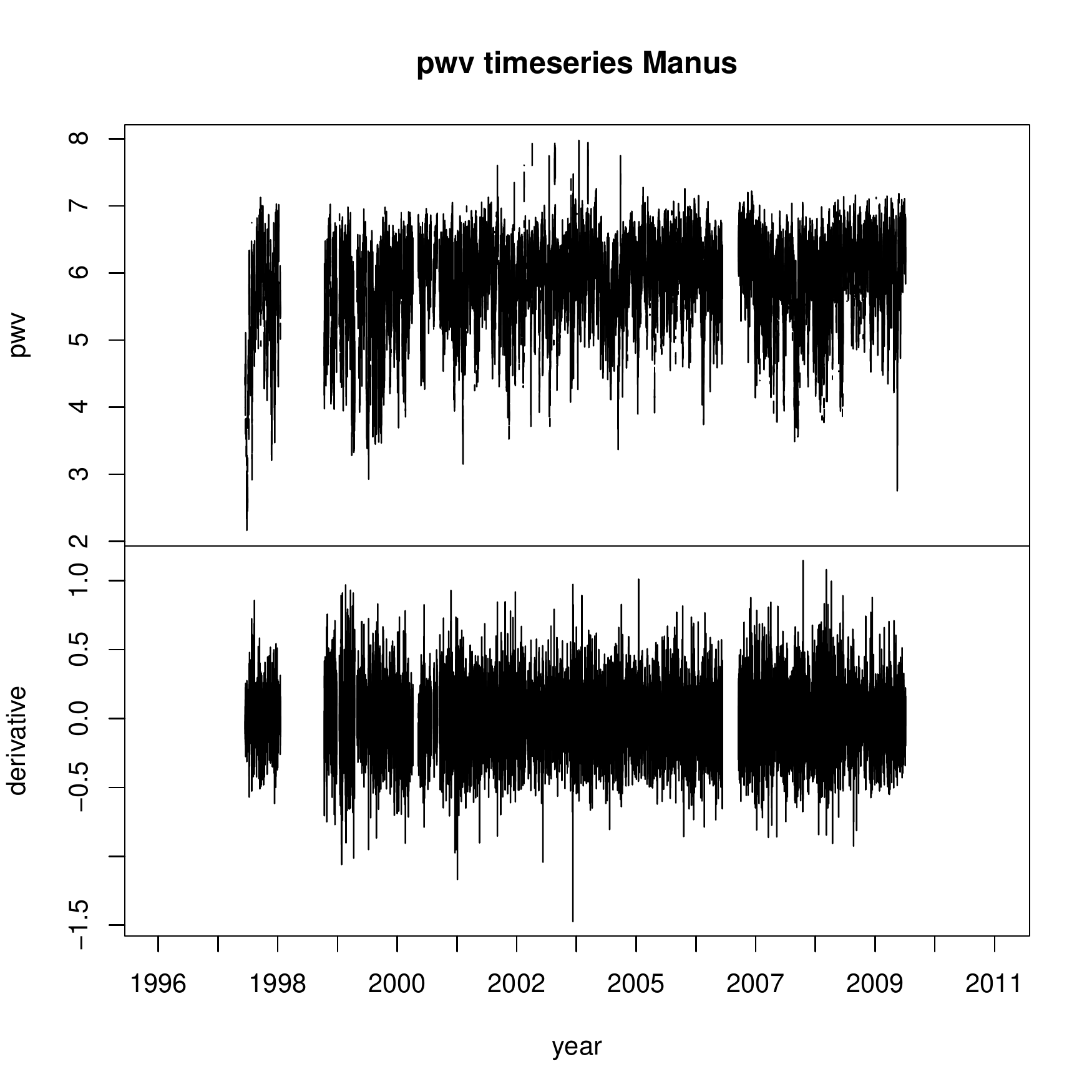}}\quad
\subfloat[][]{\includegraphics[height=8cm,page=2]{./ManusNauru1}}\\
\end{center}
\caption{$16$ years of precipitable water vapor data of the measurement station on Manus island 
and the respective data increments (a), the same for Nauru island (b) .}
\label{fig:ManusNauruDarwin}
\end{figure}
 
\hfill\\
\textit{$\lra$ Position of Figure 10}\\

\begin{figure}[ht]
\begin{center}
\subfloat[][]{\includegraphics[height=8cm,page=1]{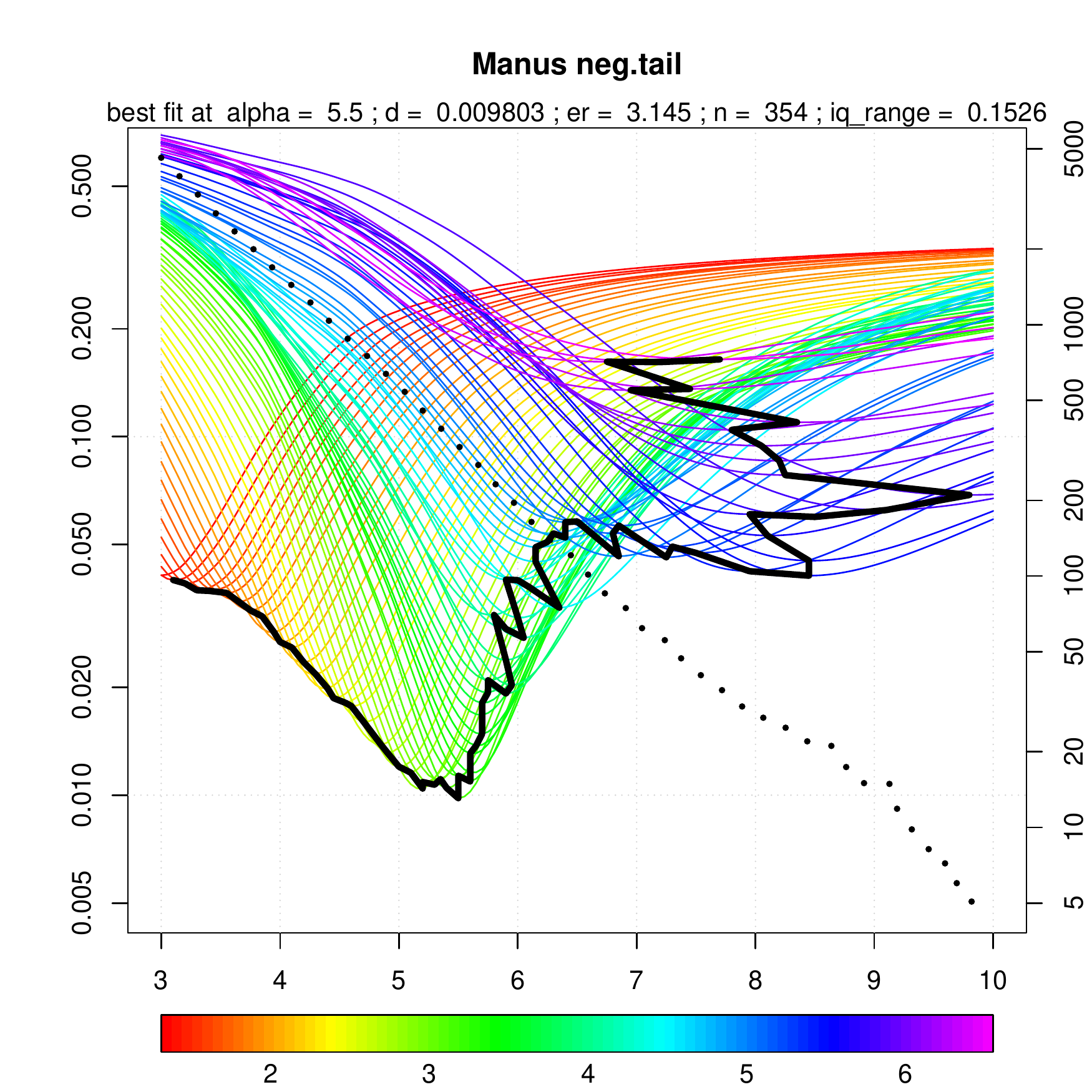}}
\subfloat[][]{\includegraphics[height=8cm,page=2]{./Manus_epsilon_sens}}\\
\subfloat[][]{\includegraphics[height=8cm,page=1]{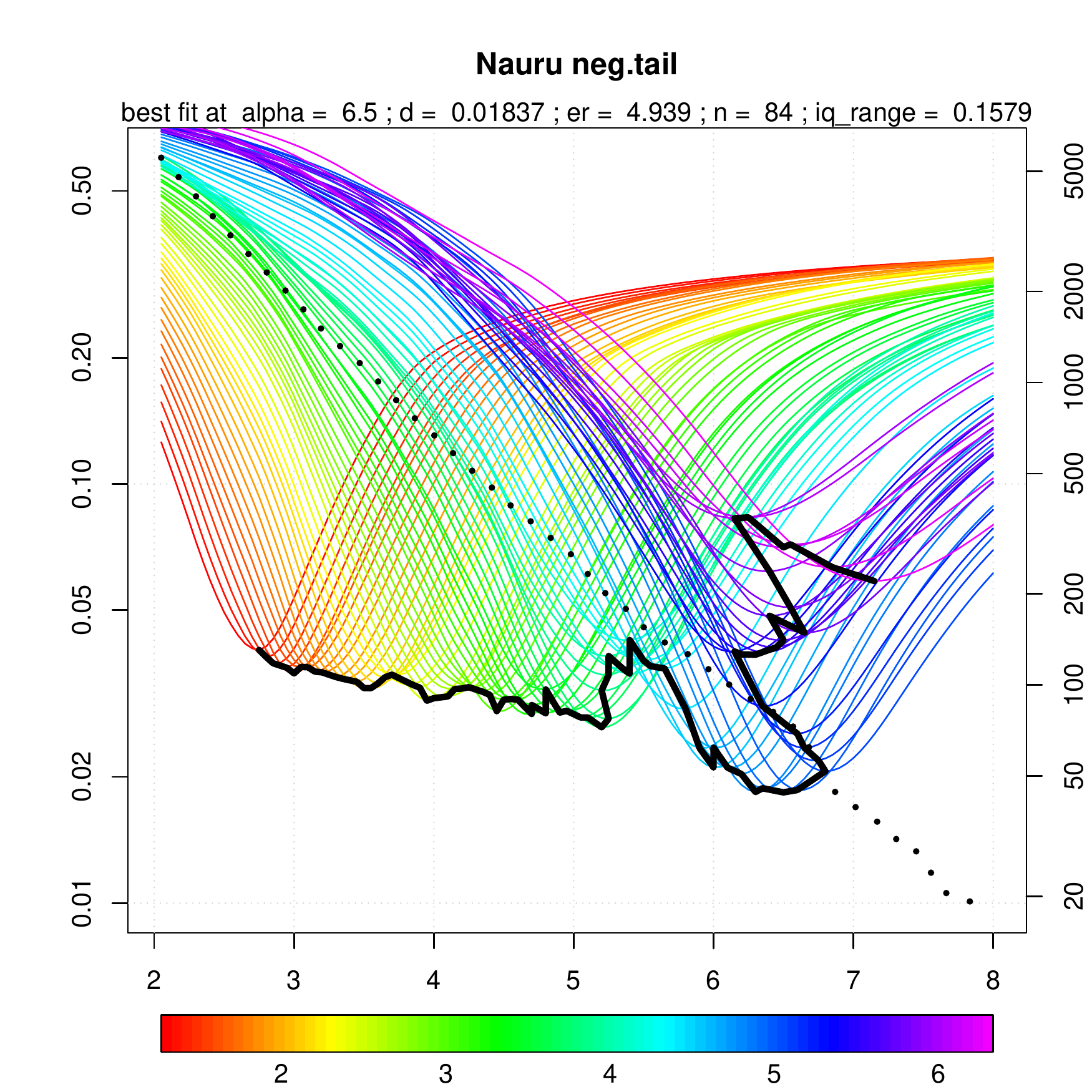}}
\subfloat[][]{\includegraphics[height=8cm,page=2]{./Nauru_epsilon_sens}}\\
\end{center}
\caption{Plots of the functions 
$\alpha \mapsto \wnt(\al; \al_0, \rho)(\om)$ the negative and positive measure $\nu^{\al_0}$ for different parameters $\rho, \alpha_0$ of the Manus island data (a) and (b) and Nauru island (c) and (d)}
\label{fig:ManuNauru_epsilon_sens}
\end{figure}

Time series of PWV and its increments from Manus and Nauru are shown in Figure 9. 
Unlike the time series at Manus, that at Nauru shows evidence of low 
frequency variations (in the local mean for PWV and in the amplitude of 
fluctuations for the increments). 
Plots of the functions $\alpha \mapsto \wnt(\al; \al_0, \rho)(\om)$ 
for the positive and negative tails of the increments distributions for both locations are shown in Figure 10. 
For the negative tail at Manus, the locus of curve minima over individual cut-offs $\rho$ takes 
a clear minimum at $\alpha_{0} = 5.5$ (corresponding to a cut-off of $\rho = 3.2$ and $\sim$150 data points), 
beyond which it rises with increasing $\rho$. The lateral meandering for $\rho \gtrsim 5$ is interpreted 
as resulting from the small number of data points entering the computation. 
In contrast, the locus of curve minima for the positive tail 
at Manus shows no evidence of a robust global minimum value.
Similar results are found at Nauru. The locus of curve minima for the 
negative tail shows a global minimum at $\alpha_{0} = 6.5$ 
($\rho \simeq 4.5$, $\sim 100$ data points), beyond which it increases. 
No evidence of a minimum is seen for the positive tail. 
The global minimum of the locus of curve minima for the 
negative tail is not as pronounced as at Manus, perhaps because 
of the low-frequency variability present in the PWV time series at Nauru.
These results show strong evidence of polynomial behavior in the negative tail of PWV, consistent with the effects of episodic
drying of the water column by convective precipitation events. 
No such threshold process produces jumps increasing PWV, and correspondingly no evidence
is found of polynomial tails for the positive increments. 
The difference in the values of $\alpha_{0}$ obtained for negative increments at the two locations (5.5 at Manus 
and 6.5 at Nauru) could reflect real differences in moist processes 
at the two stations or could simply be a result of sampling variability due to the
relatively short time series. The observed universality of precipitation event size 
distributions in \cite{PDCN10} suggests that sampling variability is likely a
large contributor to this difference. Finally, as with the ice core data considered above, 
the detection of polynomial tails in the PWV driving process with $
\alpha_{0} > 2$ indicates that the separation between fast and slow timescales in 
PWV dynamics is not sufficiently large for the central limit theorem to render
this driving process effectively Gaussian.

\section{Conclusion} 

This study presents a new method in time series analysis 
for the assessment of the proximity of data to power-law L\'evy diffusions. 
Our method is based on the notion of coupling distance recently introduced 
in the mathematics literature  
in order to measure the distance between such models on path space 
(Theorems \ref{Alex-theor} and \ref{Alex-theor-2}). 
The underlying statistical analysis involves a modification 
of the empirical Wasserstein distance, 
which is easily implementable and has favorable properties 
such as good asymptotic convergence rates as we have confirmed
by extensive simulation studies for different data lengths, tail exponents and cutoff thresholds. 
In particular, this statistic is consistent with the Wasserstein distance, 
whenever the latter exists. 

We stress that while this method provides 
insight into the noise structure of the large increments 
in the underlying data even on a pathwise 
level, it does not resconstruct 
the deterministic forcing $f$. 
For this purpose other methods 
analyzing the small 
data increments have to be applied. 
The construction of confidence intervals 
for the tail exponent is subject to future research. 
It is also remarkable that this method worked out in detail for 
the one-parameter family of power-law tail exponents 
can be adapted to other one-parameter families of L\'evy measures. 

The coupling distance curves 
introduced in this study 
provide a device for intuitive visual inspection 
in order to detect a power-law behavior of the 
driving jump noise. 
It is also shown how the characteristic structures  
of Gaussian and power-law processes 
are clearly distinguishable. 
In particular, this approach allows the detection of the 
presence of $\alpha$-stable signals, 
where the Wasserstein distance is not defined. 

The central statistical estimate of our method compares the empirical quantile 
function of the truncated increments of the time series to the quantile function 
of the candidate process being tested.  While other more intuitive approaches 
to this comparison (such as curve fitting using the empirical histogram of increments) 
may be simpler, our approach has the benefit of being rigorous, providing rate of convergence 
results, and of generalizing in a straightforward way to a broad class of stochastic 
differential equations (e.g. multiplicative or non-autonomous noise coefficients), 
driving processes (e.g. not necessarily alpha-stable) and parameters of interest 
(e.g. a skewness parameter instead of the tail exponent). 

The concept of coupling distance curves relies heavily on the 
knowledge of the explicit minimizer of the Wasserstein distance 
in one dimension. In higher dimensions, however,  
the optimal coupling is not known in closed form.  
Therefore any coupling will only provide upper bounds. 
Furthermore, it is worth noting that coupling distances were developed 
for models with additive noise (see the modeling assumptions 
in Section \ref{sec: procedure}). 
To treat more general stochastic differential equation 
models with multiplicative noise dependence 
the notion of transportation distance 
has been introduced \cite{GHK16}. 
Its statistical implementation is subject to further study. 

Our method was finally applied to an often-studied set of paleoclimate data 
and confirms our previous results in \cite{GHKK16}, where 
the tail exponents were determined by a non-systematic application of coupling distances. 
In addition, we systematically 
analyze a set of atmospheric data 
from the Western Tropical Pacific  
and detected polynomial tail exponents 
interpreted as resulting 
from the threshold behavior of precipitation. 
Future studies will generalize the approach considered 
in this study to allow analysis of time series with strong
deterministic non-stationarities such as annual or diurnal cycles.

\bigskip
\section*{Acknowledgements}

The first three authors would like to thank the International Research Training Group 1740 Berlin-S\~ao Paulo: 
Dynamical Phenomena in Complex Networks: Fundamentals and Applications, the group of Sylvie Roelly at University Potsdam 
and Alexei Kulik from the Ukrainian National Academy of Sciences. MAH acknowledges the support 
of the FAPA grant ``Stochastic dynamics of L\'evy driven systems'' of Universidad de los Andes. 
AHM acknowledges support from the Natural Sciences and Engineering Research Council (NSERC) of Canada.

\end{document}